\documentclass{article}

\usepackage[backend=bibtex, 
sortcites=true,
firstinits=true,
hyperref,
maxbibnames=99,
]{biblatex}

\bibliography{Bibliography}{}

\usepackage{enumerate}

\usepackage[centertags]{amsmath}
\usepackage{amsfonts}
\usepackage{amssymb}
\usepackage{amsthm}
\usepackage{newlfont}
\usepackage{mathtools}
\usepackage[top=1in, bottom=1in, left=1in, right=1in]{geometry}
\usepackage{tikz}

\theoremstyle{plain}
\newtheorem{theorem}{Theorem}[section]
\newtheorem{proposition}[theorem]{Proposition}
\newtheorem{corollary}[theorem]{Corollary}
\newtheorem{lemma}[theorem]{Lemma}

\theoremstyle{definition}
\newtheorem{definition}[theorem]{Definition}

\newtheorem{example}{Example}[section]
\newtheorem*{remark}{Remark}

\newcommand{\R}{\mathbb{R}}
\newcommand{\N}{\mathbb{N}}
\newcommand{\Z}{\mathbb{Z}}
\newcommand{\C}{\mathbb{C}}

\newcommand{\Deg}{\operatorname{Deg}}

\newcommand{\sgn}{\operatorname{sgn}}

\newcommand{\eps}{\varepsilon}

\newcommand{\Hidden}[1]{}

\begin{document}

\title{Every Salami has two ends}
\author{
Bobo Hua%\footnote{School of Mathematical Sciences, LMNS, Fudan University, Shanghai 200433, China;  Shanghai Center for Mathematical Sciences, Jiangwan Campus, Fudan University, No. 2005 Songhu Road, Shanghai 200438, China.
%}
~~~~~
%~~~~~J\"urgen Jost\footnote{MPI Leipzig, jost@mis.mpg.de}~~~~~
%~~~~~
Florentin M\"unch%\footnote{MPI MiS Leipzig, muench@mis.mpg.de}%~~~~~Christian Rose\footnote{MPI Leipzig, crose@mis.mpg.de}
}
\date{\today}
\maketitle

\begin{abstract}
A salami is  a connected, locally finite, 
weighted graph with non-negative Ollivier Ricci curvature and at least two ends of infinite volume. We show that every salami has exactly two ends and no vertices with positive curvature. We moreover show that every salami is recurrent and admits harmonic functions with constant gradient.
The proofs are based on extremal Lipschitz extensions, a variational principle and the study of harmonic functions.
Assuming a lower bound on the edge weight, we prove that salamis are quasi-isometric to the line, that the space of all harmonic functions has finite dimension, and that the space of subexponentially growing harmonic functions is two-dimensional. Moreover, we give a Cheng-Yau gradient estimate for harmonic functions on balls.
\end{abstract}

\tableofcontents

%%%%%%%%%%%%%%
%\section{Introduction}
%%%%%%%%%%%%%%
%\tableofcontents

\section{Introduction}
An ancient saying states ``Everything has an end, only the salami has two." We will put this saying into a mathematical framework.

\begin{theorem}\label{thm:Main}
Let $G=(V,w,m,d)$ be a salami, i.e., a graph with non-negative Ollivier curvature and at least two ends of infinite volume. Then,

\begingroup
\vspace{0.8mm}
\setlength{\tabcolsep}{8pt} % Default value: 6pt
\renewcommand{\arraystretch}{1.4} % Default value: 1
\begin{tabular}{ c l }
 (Theorem~\ref{thm:ends})  &  $G$ has exactly two ends, \\
 (Corollary~\ref{cor:flat})&  $G$ is flat,   \\
 (Theorem~\ref{thm:recurrent})& $G$ is recurrent,   \\
 (Theorem~\ref{thm:constGradient})&  There exists a harmonic function $f \in \R^V$ with constant gradient. %$\nabla_-f=\nabla_+f=1$.
\end{tabular}
\endgroup
\end{theorem}
%The definition of a salami is given in Definition~\ref{def:salami}.
%The definition of flatness is given in Definition~\ref{def:flat}.
%The definition of $\nabla_-$ and $\nabla_+$ is given in Definition~\ref{def:gradients}.

We summarize the proof ideas. 
Suppose there are two ends $X,Y \subset V$ of infinite measure, separated by a finite $K \subset V$.
We consider the space $\mathcal F$ of functions coinciding, when fixing the function on $K$, with their maximal 1-Lipschitz extension on $Y$, and with their minimal 1-Lipschitz extension on $X$.
Using non-negative curvature and infinite measure of $X$ and $Y$, one can show that there are harmonic functions in $\mathcal F,$ see Lemma~\ref{lem:harmonicEverywhere}. This proves the existence of nonconstant Lipschitz harmonic functions on the graph, see \cite{kleiner2010new,tao2010gromov} for the results on Cayley graphs of groups. The existence of such harmonic functions has strong implications for the geometry, see e.g.  \cite{cheeger1995linear}
 for Riemannian manifolds with nonnegative Ricci curvature. We then show that the set $\mathcal H$ of harmonic functions in $\mathcal F$ stays invariant when varying the partition $(X,Y,K)$ appropriately, see Lemma~\ref{lem:HsameFinite} and Lemma~\ref{lem:SameH}. 
When supposing a third end, one can assign it to $X$ or to $Y$, and both assignments have the same function space $\mathcal H$.
From this, we conclude that the corresponding maximal and minimal Lipschitz extensions on the third end must coincide, which we apply to show that there are exactly two ends and that the gradient of harmonic functions in $\mathcal F$ is constant one. From this, we deduce  flatness and recurrence.
As a corollary, we obtain the following.
\begin{corollary}Let $G=(V,w,m,d)$ be a graph with non-negative Ollivier curvature. If $\inf_{x\in V}m(x)>0,$ then the number of ends of $G$ is at most two.
\end{corollary}

We can get even stronger results when assuming a lower bound on the edge weight.
\begin{theorem}
Let $G=(V,w,m,d)$ be a salami where $d$ is the combinatorial distance. Suppose there is $\eps>0$ such that $w(x,y) \geq \eps$ for all $(x,y) \in E$. Then,

\begingroup
\vspace{0.8mm}
\setlength{\tabcolsep}{8pt} % Default value: 6pt
\renewcommand{\arraystretch}{1.4} % Default value: 1
\begin{tabular}{ c l }
 (Theorem~\ref{thm:quasiIsom})  &  $G$ is quasi-isometric to the line, \\
  (Theorem~\ref{thm:finiteDimention})& The space of all harmonic functions has finite dimension,    \\
 (Corollary~\ref{cor:dimSubexpGrowth})&  The space of subexponentially growing harmonic functions is two-dimensional.  
\end{tabular}
\endgroup
\end{theorem}
We also summarize the proof ideas. For the quasi isometry, we show that level sets of harmonic functions in $\mathcal F$ are of uniformly bounded cardinality and connected. Connectedness is proved via a topological covering argument using the result that a salami has two ends. Uniformly bounded size is proved using the constant gradient of harmonic functions and the lower bounded edge weight. For the dimension of subexponentially growing functions, we show that the variance of a harmonic function along a level set of a harmonic function in $\mathcal F$ must grow at least exponentially by a recursion estimate.

When furthermore assuming bounded geometry, we prove the Cheng-Yau gradient estimate for positive harmonic functions, see Theorem~\ref{thm:ChengYau}.

\subsection{Interpretation of the main results}

Discrete differential geometry has been a vibrant research subject in the last decade. 
Particularly, various different discrete analogs of Ricci curvature have attracted notable interest. For graphs, there are three different basic Ricci curvature notions. The Bakry-\'Emery curvature is based on a discrete Bochner formula \cite{schmuckenschlager1998curvature, lin2010ricci}, the Ollivier curvature is based on the Wasserstein distance of balls \cite{ollivier2007ricci,ollivier2009ricci, jost2014ollivier}, and the entropic Ricci curvature is based on convexity of the entropy \cite{mielke2013geodesic, erbar2012ricci}.
 
Let us briefly review the known consequences of non-negative Ricci curvature on graphs. We remark that the Bakry-\'Emery curvature is the only one for which all following properties hold true. 
\begin{itemize}
\item Gradient estimates \cite{hua2017stochastic, keller2018gradient,gong2017equivalent, lin2015equivalent, munch2017ollivier,erbar2018poincare}  for the heat semigroup $P_t$ for a suitable norm $\|\cdot\|$,
\[ \|\nabla P_t f\| \leq \| \nabla f\|. \]
\item Eigenvalue  estimates \cite{lin2010ricci,munch2019non, erbar2018poincare} for the the first nontrivial eigenvalue of $-\Delta$ in terms of the diameter,
\[
\lambda_1 \geq \frac C{(\mathrm{diameter})^2}.\]
\item Bounded Liouville property
 \cite{hua2019liouville,jost2019Liouville}, i.e., every bounded harmonic function is constant.
 \item Non-existence of expander graphs \cite{munch2019non,salez2021sparse}.
\item Polynomial volume growth and non-linear gradient estimates for $P_t$, see \cite{bauer2015li,munch2019li,horn2014volume,dier2017discrete},
\item Trivial first homology in case of positive curvature somewhere \cite{munch2020spectrally, kempton2017relationships}. 
\end{itemize}
 
A discrete Cheeger-Gromoll splitting theorem is still missing on the list. The classical Cheeger-Gromoll splitting theorem states that every complete Riemannian manifold with non-negative Ricci curvature which possesses a straight line splits out a factor $\R$, see
\cite{cheeger1971splitting}. As a consequence, any manifold $M$ with non-negative Ricci curvature has at most two ends. If it has two ends, then $M=N\times \R$ where $N$ is a compact manifold.
The proof in the Riemannian case relies on Busemann functions and on the Laplacian comparison theorem requiring  finiteness of space dimension.
It is proven in \cite{cai1995gap} that manifolds allowing for some negative Ricci curvature within a small compact set also have at most two ends.

For weighted manifolds with an infinite-dimensional curvature condition, several splitting theorems have been shown \cite{fang2009two,su2012rigidity,munteanu2014geometry}.
It is well known that a general splitting cannot be true as hyperbolic spaces can be made non-negatively curved when giving them an appropriate weight. However,
 it was proven that a non-negatively curved weighted manifold can have at most two ends if the weight grows at most sublinearly \cite[Theorem~1.1]{fang2009two}.
The case of arbitrary weights but infinite volume on two ends seems still open for weighted manifolds.

In the discrete case, without a very restrictive notion of a discrete straight line, it seems  still far away to prove a general splitting theorem for graphs. Since there are various counterexamples to strict splitting on graphs. The best one could hope for is splitting in the sense of quasi isometries, see the Examples section.
As manifolds with two ends always possess straight lines, it is natural to ask what kind of splitting one can obtain for graphs with two ends and non-negative Ricci curvature. This article is dedicated to answer this question for Ollivier curvature.

\subsection{History of proof attempts}
In the last few years, several authors tried to establish a discrete Cheeger-Gromoll theorem, usually without success due to the potential counterexamples. 
The best results so far in this direction are splitting theorems for groups and their Cayley graphs \cite{nguyen2019cheeger,bar2020conjugation}.
However, it is still open whether a splitting theorem holds for general graphs.

In March 2017, the authors of this article had a lunch break together at a conference in Sanya, and
after seeing the counterexamples for a general discrete Cheeger-Gromoll splitting theorem, the first author asked whether non-negatively graphs must have at most two ends.

It was clear from the beginning that one should try to find harmonic functions growing like the distance, so that one can deduce two ends from the maximum principle.
The first attempts to find these harmonic functions were using discrete Busemann functions and later, the distance from a minimal cut. Although the second author was very sure to have found a proof via the minimal cut approach and even announced a talk on this alleged result, it turned out that the naive Busemann function or minimal cut approach can never work as the desired harmonic functions need not be integer-valued, see Example~\ref{ex:noIntegerH0}. In the next attempt to prove two ends, the Liouville property was shown by accident \cite{jost2019Liouville}.

The first more or less successful attempt was using a brute force approach to construct harmonic functions. By the maximum principle for the gradient and appropriate normalization, it was shown that the constructed functions must be Lipschitz, but it was still possible that the functions grow too slowly. Therefore, the maximal Lipschitz extension was applied at infinity to construct Lipschitz sharp superharmonic functions. Doing this from both ends and showing recurrence showed that these functions must be harmonic. From this, two ends followed easily. However, the proof of the existence of the desired Lipschitz harmonic function was very intricate, e.g., it was not even clear that the harmonic function must have finite level sets. Moreover, strong assumptions on the weights of the graph were necessary, so the authors slowed down on this approach.

Finally during an unsuccessful attempt to prove a log-Sobolev inequality, the authors came up with the idea of applying the extremal Lipschitz extension not at infinity, but from the inside, and then noticed that one can even ensure that the Laplacian is constant inside, see Lemma~\ref{lem:constLaplace}. And this idea is what this article builds on.

\subsection{Structure of the article}

In Section~\ref{sec:Setup}, we explain the setup and give some general facts about Ollivier curvature. In Section~\ref{sec:SalamiPartitions}, we prove Theorem~\ref{thm:Main}. In Section~\ref{sec:levelSets}, we show splitting in the sense of quasi-isometries. In Section~\ref{sec:harmonic}, we investigate the space of harmonic functions. Finally in Section~\ref{sec:Examples}, we give examples to justify our assumptions, i.e., that the ends have infinite volume, and for some applications, that the edge weight is bounded from below.

\section{Setup, notations and basic facts} \label{sec:Setup}

A \emph{graph} $G=(V,w,m)$ consists of a countable set $V$, a symmetric function $w:V\times V \to [0,\infty)$ called \emph{edge weight} with $w=0$ on the diagonal, and a function $m:V\to (0,\infty)$ called vertex measure or volume.
We write $x\sim y$ if $w(x,y)>0$. The vertex degree $\Deg(x)$ for $x \in V$ is given by $\Deg(x):=\frac 1 {m(x)}\sum_{y\in V} w(x,y).$
The edge weight $w$ induces a symmetric edge relation $E=\{(x,y):w(x,y)>0\}$. We write $E(X,Y):=E \cap (X\times Y$) for $X,Y \subset V$. We denote by $|\cdot|$ the cardinality of a set. A graph is called \emph{locally finite} if $|\{y:y\sim x\}|<\infty$ for all $x\sim y$.

We say a metric $d:V^2 \to[0,\infty)$  is a path metric on a graph $G$ if
\[
d(x,y) = \inf\left\{\sum_{i=1}^n d(x_{i-1},x_i): x=x_0 \sim \ldots \sim x_n=y \right\}.
\]
We will always assume that all balls are finite.
We say $G=(V,w,m,d)$ is a graph if $(V,w,m)$ is a locally finite graph and $d$ a  path metric on $(V,w,m)$ with finite balls.
The discrete Hopf-Rinow theorem from  \cite{keller2019new} shows that finiteness of balls is equivalent to completeness of $d$ which is again equivalent to geodesic completeness, i.e., every infinite geodesic has infinite length.

Let $d_0: V \times V \to \N_0$ be the combinatorial graph distance given by
\[
d_0(x,y) := \inf \{n: x=x_0 \sim \ldots \sim x_n =y\}.
\] For $W \subset V$ and $R \in [0,\infty)$, we write $B_R(W) := \{x\in V : d(x,W) \leq R\}$. We also write $B_R(x)=\{y:d(x,y) \leq R\}$ and $S_R(x) = \{y:d(x,y) = R\}$.

A subset $W \subseteq V$ is called connected if $d(x,y)<\infty$ for all $x \sim y \in W$. This means a subset $W \subseteq V$ is  connected if the induced subgraph on $W$ is connected.
%Let $G=(V,w,m)$ be a graph. 
A connected component of $W\subseteq V$ is a maximal connected subset of $W$.
\begin{definition}
A graph $G=(V,w,m)$ is said to have at least $n$ ends, if there exists a finite $K\subset V$ such that $V\setminus K$ has at least $n$ infinite connected components. We say $G$ has exactly $n$ ends, if $G$ has at least $n$ ends but not at least $n+1$ ends.
\end{definition}

Let $G=(V,w,m)$ be a graph. The Laplace operator $\Delta: \R^V \to \R^V$ is given by
\[
\Delta f(x) := \frac 1 {m(x)}\sum_{y} w(x,y) (f(y)-f(x)).
\]
A function $f \in \R^V$ is called harmonic if $\Delta f = 0$.
The function space $C_c(V)$ is given by $$C_c(V) := \{f \in \R^V : f(x)= 0 \mbox{ for all but finitely many } x \in V\}.$$ For $f \in \R^V$, we write $\|f\|_\infty := \sup_{x \in V} |f(x)|$.
For any $\Omega\subset V,$ we denote by $1_\Omega$ the indicator function on $\Omega.$

For $f \in C_c(V)$ and $g \in \R^V$, we write 
\[
\langle f, g \rangle := \sum_{x \in V}f(x)g(x)m(x).
\]
We notice that $\langle f, \Delta g \rangle = \langle \Delta f, g \rangle$ for $f \in C_c(V)$ and $g \in \R^V$.

\subsection{Ollivier curvature}
Let $G=(V,w,m,d)$ be a locally finite graph with path metric $d$.
According to \cite{munch2017ollivier}, the Ollivier curvature $\kappa(x,y)$ for $x, y \in V$ with $R:=d(x,y)> 0$ is given by
\[
\kappa(x,y) := \frac 1 R \inf_{\substack{f(y)-f(x)=R\\f \in Lip(1)}} \Delta f(x) - \Delta f(y)
\]
with $Lip(1):= \{f\in \R^V: f(y)-f(x) \leq d(x,y) \mbox{ for all } x,y\in V\}$.
We set $Lip(1,K):=\{f \in \R^V:f(y)-f(x)\leq d(x,y) \mbox{ for all } x,y \in K\}$.
We remark that $\kappa$ and $Lip(1)$ depend on the choice of the path distance $d$.

\begin{definition}\label{def:flat}
We write $\kappa(x) := \inf_{y\sim x} \kappa(x,y)$ for $x \in V$.
A graph $G=(V,w,m,d)$ is called \emph{flat} if $\kappa(x)=0$ for all $x \in V$.
\end{definition}

It is shown in \cite{munch2017ollivier, veysseire2012coarse} that a lower curvature bound is equivalent to a gradient estimate for the continuous time heat equation.
We now show that a lower curvature bound also implies a gradient estimate for the time discrete random walk. This is well known 
for the combinatorial graph distance, see \cite{loisel2014ricci,bourne2018ollivier,lin2011ricci}.
Here, we give a proof in the case of a general path metric.

\begin{lemma}\label{lem:DiscreteTimeLipContraction}
Let $G=(V,w,m,d)$ be a graph and $x, y \in V$.
Let $f \in Lip(1)$. Let $0<\eps \leq \frac 1 {\Deg(x) + \Deg(y)}$ and
let $H:=id + \eps \Delta$.
Then,
\[
|Hf(y)-Hf(x)| \leq d(x,y)(1-\eps \kappa(x,y)).
\]
\end{lemma}

\begin{proof}

Let $\delta_v := 1_v/m(v)$ and 
\[
m_v(w) := \langle H 1_w,\delta_v \rangle.
\]
We see that $m_x$ and $m_y$ are probability measures, particularly, $m_x(x), m_y(y) \geq 0$ as $\eps$ is small.
We note that $\int g dm_x = Hg(x)$ for all $g \in \R^V$.

We write $q(x,y) := w(x,y)/m(x)$.
Due to Kantorovich duality, we can write
\[
\kappa(x,y)= \sup_\rho \sum_{x',y'} \rho(x',y') \left( 1- \frac{d(x',y')}{d(x,y)} \right),
\]
where the supremum is taken over all $\rho:V^2 \to [0,\infty)$ satisfying $\sum_{x'} \rho(x',y')= q(y,y')$ for all $y' \in V \setminus y$ and $\sum_{y'} \rho(x',y')=q(x,x')$ for all $x' \in V \setminus x$, see \cite[Proposition~2.4]{munch2017ollivier}.
%Then, every such $\rho$ satisfies $\sum_{y'} \rho(x,y') \leq \rho(x,y)+ \Deg(y)$ and $\sum_{x'} \rho(x',y) \leq \rho(x,y)+ \Deg(x)$.
Let the supremum be attained at $\rho$.
We write 
\[\rho'(x',y') := \eps \cdot \begin{cases} \frac 1 \eps + \rho(x,y)- \sum_{x',y'}\rho(x',y') &: x'=x,y'=y, \\
\rho(x',y') &: \mbox{ otherwise}
\end{cases}
\]
which is non-negative as $\frac 1 \eps \geq \Deg(x)+ \Deg(y)$.
Moreover, $\sum_{x'}  \rho'(x',y') = m_y(y')$ for all $y' \in V$ and
$\sum_{y'} \rho'(x',y') = m_x(x')$ for all $x'\in V$. Thus,
\[
\eps \kappa(x,y)= \sum_{x',y'} \rho'(x',y') \left( 1- \frac{d(x',y')}{d(x,y)} \right).
\]
As $\sum_{x',y'} \rho'(x',y')=1$, we have
\[
\sum_{x',y'} \rho'(x',y') d(x',y') = d(x,y)(1-\eps\kappa(x,y)).
\]
Hence again by Kantorovich duality for the $\ell_1$ Wasserstein distance $W_1$,
\[
Hf(x)-Hf(y) \leq \sup_{g \in Lip(1)} \int g dm_x - \int g dm_y = W_1(m_x,m_y) = \inf_{\widetilde \rho} \sum_{x',y'} \widetilde \rho(x',y') d(x',y') \leq d(x,y)(1-\eps\kappa(x,y))
\]
where the infimum is taken over all non-negative $\widetilde \rho$ with $\sum_{x'}\widetilde \rho(x',y')=m_y(y')$ for $y' \in V$ and $\sum_{y'}\widetilde \rho(x',y')=m_x(x')$ for $x' \in V$.
This finishes the proof.
\end{proof}

We now compute the curvature of trees with combinatorial distance.
This is well known for unweighted graphs, see e.g. \cite{jost2014ollivier}. We give the proof for weighted graphs.

\begin{lemma}\label{lem:CurvatureTrees}
Let $G=(V,w,m,d_0)$ be a tree, i.e., $G$ is connected and $G\setminus \{x\}$ is disconnected for all $x \in V$. Let $x\sim y \in V$. Then, 
\[
\kappa(x,y) = 2(q(x,y) + q(y,x)) - \Deg(x)-\Deg(y).
\]
where
\[
q(x,y):=\frac{w(x,y)}{m(x)}.
\]
\end{lemma}

\begin{proof}
We first show ``$\leq$".
We write $X=B_1(x)\setminus B_1(y)$ and $Y=B_1(y) \setminus B_1(x).$
As $G$ is a tree, there exists $f \in Lip(1)$ with 
\[
f(z)=\begin{cases}
-1&:z \in X\\
0&:z = x\\
1&:z=y\\
2&: z \in Y.
\end{cases}
\]
Hence,
\[
\kappa(x,y) \leq \Delta f(x) - \Delta f(y) = ( 2q(x,y) - \Deg(x))-(\Deg(y) - 2q(y,x))
\]
showing ``$\leq$".

We finally prove ``$\geq$".
Let $f \in Lip(1)$ with $f(y)-f(x)=1$. Then, $f \leq f(y)+1$ on $Y$ and $f \geq f(x)-1$ on $X$ giving $$\Delta f(y) \leq \Deg(y) - 2q(y,x) \quad \mathrm{and}\quad \Delta f(x) \geq 2q(x,y) - \Deg(x).$$ Taking infimum over $f$ shows
$$\kappa(x,y) \geq 2(q(x,y) + q(y,x)) - \Deg(x)-\Deg(y).$$
This shows ``$\geq$" and finishes the proof.
\end{proof}

We now define several gradient notions.
\begin{definition}\label{def:gradients}
Let $G=(V,w,m,d)$ be a connected locally finite graph.
For $f \in \R^V$ and $x \in V,$ let $$\nabla_+ f(x) := \max_{y\sim x} \frac{f(y)-f(x)}{d(x,y)},\quad\nabla_-f(x):= \max_{y\sim x} \frac{f(x)-f(y)}{d(x,y)},\quad \mathrm{and}$$
$$|\nabla f|(x) := \max_{y\sim x} \frac{|f(y)-f(x)|}{d(x,y)}.$$
\end{definition}

In general, non-negative Ricci curvature has strong implications for the behavior of harmonic functions. In particular, a bounded Liouville property was proven for graphs with non-negative Ricci curvature \cite{hua2019liouville,jost2019Liouville}. We now give a maximum principle for the gradient of harmonic functions by using methods from \cite{jost2019Liouville, munch2019non}.
For convenience, we give the result only in terms of the combinatorial distance $d_0$.

\begin{theorem}
Let $G=(V,w,m,d_0)$ be a locally finite graph.
Let $W\subset V$ be finite and non-empty.
Suppose $\kappa(x,y)\geq 0$ for all $x,y \in W$ with $x\sim y$.
Let $u:B_2(W)\to \R$ satisfy $\Delta u(y)\geq \Delta u(x)$ whenever $x, y \in W$ with $x \sim y$ and  $u(y) \geq u(x)$.
Then,
\[
\max_W |\nabla u| \leq \max_{S_1(W)} |\nabla u|.
\]
\end{theorem}

\begin{proof}
Let $\widetilde G$ be the graph on $B_2(W)$ with 
$\widetilde w(x,y) = 0$ if $x,y \in S_2(W)$ and $\widetilde w(x,y)=w(x,y)$ otherwise. Let $\widetilde d$ be the combinatorial distance on $\widetilde G$.
Then, $\widetilde d \geq d:=d_0$. 
We claim that for all $x\sim y$ with $x,y \in W$ and all $x'\in B_1(x)$ and $y' \in B_1(y)$, we have  $\widetilde d(x',y') = d(x',y')$.
For seeing this, we notice that $\widetilde d(x',y') \leq 3$, and every path from $x'$ to $y'$ in $G$ of length at most two has to stay in $B_2(W),$ and cannot contain any edge between vertices in $S_2(W).$ Hence, $\widetilde d(x',y') = d(x',y')$ proving the claim.
This shows that $\widetilde \kappa(x,y) = \kappa(x,y) \geq 0$ for all $x,y \in W$ with $x \sim y$.

W.l.o.g., we assume $\max_W  |\nabla u| = 1$.
We suppose the theorem is wrong, i.e., $\max_W  |\nabla u| = 1 > \max_{S_1(W)} |\nabla u|$.
Then, there exist $x,y \in W$ with $u(y)-u(x) = \widetilde d(x,y)$. 
We assume $\widetilde d(x,y)=n$ to be maximal.
Let $(x=x_0,\ldots,x_n=y)$ be a shortest path w.r.t. $\widetilde d$. By $|u(x_{i+1})-u(x_i)|\leq 1$ for $i=0,\dots, n-1,$ we have $u(x_i)=u(x)+i$ for all $ i=0,\ldots,n$.
Hence, $|\nabla u|(x_i) = 1$ showing $x_i \in W$. As $\widetilde \kappa(x_i,x_{i+1}) \geq 0$, we get $\widetilde\kappa(x,y) \geq 0$. 
By assumption of the theorem, we have $\Delta u(y) \geq \Delta u(x)$.
Thus with an optimal transport plan $\rho$,
\begin{align*}
0\leq \Delta u(y) - \Delta u(x) &= \sum_{y',x'} \rho(x',y') (u(y')-u(x')-\widetilde d(x,y)) \\
&\leq \sum_{y',x'} \rho(x',y') (\widetilde d(x',y')-\widetilde d(x,y))
\\&\leq 0,
\end{align*}
where the second inequality follows from $u(y')-u(x') \leq \widetilde d(x',y')$ as $|\nabla u| \leq 1$ on $B_1(W)$ and as every edge in a path from $x'$ to $y'$ w.r.t. $\widetilde G$ has to contain a vertex in $B_1(W)$, and the third inequality follows from $\widetilde \kappa(x,y) \geq 0$. Hence $\widetilde d(x',y') = u(y')-u(x')$ whenever $\rho(x',y') >0,$ and $\widetilde \kappa(x,y)=0$. We notice that there is an optimal transport plan $\rho$ with $\rho(x',y')>0$ for some $x',y'$ with $\widetilde d(x',y')< \widetilde d(x,y)$, see e.g.
\cite[Theorem~2.1.2]{munch2019non}.
 As $\widetilde \kappa(x,y)\leq 0$, there exists $x',y'$ with $\widetilde d(x',y') > \widetilde d (x,y)$ and $\rho(x',y')>0$. For such $x'$ and $y',$ we have $$u(y')-u(x')=\widetilde{d}(x',y')>\widetilde{d}(x,y)=n.$$ Moreover, $|\nabla u|(x')=|\nabla u|(y')=1$ implies that $x',y'\in W$. This contradicts maximality of $\widetilde d(x,y)$ and finishes the proof.
\end{proof}

Applying the theorem to harmonic functions, we obtain the following corollary.
\begin{corollary} \label{cor:maxPrinciple}
Let $G=(V,w,m,d_0)$ be a locally finite graph with non-negative Ollivier curvature. Let $x \in V$ and $u:B_{R+1}(x) \to \R$ be harmonic on $B_{R-1}(x)$.
Then,
\[
\max_{B_{R}(x)} |\nabla u| = \max_{S_R(x)} |\nabla u|.
\]
\end{corollary}

\subsection{Salamis}
We now introduce the main actor of this article.

\begin{definition}\label{def:salami}
We call $P=(X,Y,K)$ a \emph{salami partition} of a graph $G=(V,w,m)$ if 
\begin{enumerate}[(a)]
\item $X\dot \cup Y \dot \cup K= V$,
\item $K$ is finite,
\item $E(X,Y)=\emptyset$,
\item $m(X)=m(Y)=\infty$.
\end{enumerate}
We remark that $X,Y$ and $K$ may not be connected in the above definition. We call a salami partition \emph{connected}, if $K$ is connected.
We call a connected, locally finite graph $G=(V,w,m,d)$ a \emph{salami}, if it has non-negative Ollivier curvature, and if there is a salami partition. %A salami is called \emph{combinatorial} if it has non-negative curvature w.r.t the combinatorial distance, i.e., $(V,w,m,d_0)$ is a salami.

\end{definition}

\begin{lemma}\label{lem:connectedSalami} 
Let $G=(V,w,m,d)$ be a salami, and let $x \in V$. Then, there exists a connected salami partition $(X,Y,K)$ with $x \in K$. 
\end{lemma}
\begin{proof}
Let $(X_0,Y_0,K_0)$ be a salami partition. W.l.o.g., $x \in K_0$.
Let $R := \max_{v,w \in K_0} d(v,w) < \infty$.
Now let $K:=B_R(K_0)$. We notice that $K$ is connected.
One can easily check that $(X_0 \setminus K, Y_0 \setminus K,K)$ is a salami partition which finishes the proof.
\end{proof}

\section{Salami partitions, Lipschitz extensions and harmonic functions} \label{sec:SalamiPartitions}

The aim of this section is to prove Theorem~\ref{thm:Main}.
To this end, we introduce extremal Lipschitz extensions and employ variational principles to investigate Lipschitz harmonic functions.

\subsection{Extremal Lipschitz extension}

In this subsection, we introduce the extremal Lipschitz extension operator $S:Lip(1,K) \to \R^V$ and its image $\mathcal F$. 
We recall $Lip(1,K)=\{f \in \R^V: f(y)-f(x) \leq d(x,y) \mbox{ for all } x,y \in K\}$.
We will give a variational principle to show that $\mathcal F$ contains harmonic functions.
\begin{definition}
Let $P=(X,Y,K)$ be a salami partition. 
For $f \in Lip(1,K)$, we set
\[
S(P)f(v) := Sf(v):= \begin{cases}
f(v)&:v \in K\\
\sup_{w \in K} f(w)-d(v,w) &: v \in X\\
\inf_{w \in K} f(w)+d(v,w) &: v \in Y.
\end{cases}
\]

We define
\[
\mathcal F := \mathcal F(P) := S(P)(Lip(1,K))
\]
and 
\[
\mathcal H := \mathcal H(P) := \{f \in \mathcal F(P): \Delta f = 0\}.
\]
\end{definition}

We will later show that $\mathcal H$ is non-empty.
We next study some basic properties of the extremal Lipschitz extension operator $S$ and its image $\mathcal F$.

\begin{lemma}\label{lem:Sproperties}
Let $G=(V,w,m,d)$ be a graph with a salami partition $P=(X,Y,K)$. Let $f,g \in Lip(1,K)$. Then,
\begin{enumerate}[(1)]
\item $S^2 f = Sf$,
\item $ \mathcal F \subseteq Lip(1)$,
%\item $\mathcal F  = f\in Lip(1) : $
\item $S f \leq Sg$ whenever  $f|_K \leq g|_K $,
\item For every $x \in X$, there exists $v \in K$ with $Sf(v)-Sf(x)=d(v,x)$,
\item For every $y \in Y$, there exists $v \in K$ with $Sf(y)-Sf(v)=d(v,x)$,
\item $Sf = \inf\{h \in Lip(1):h|_K = f|_K  \}$ on $X\cup K$,
\item $Sf = \sup\{h \in Lip(1):h|_K = f|_K  \}$ on $Y\cup K$,
\item $\mathcal F = \{h \in Lip(1) : (\nabla_+h)|_X =1, \;  (\nabla_-h)|_Y =  1, \; \lim_Y h = \infty ,\; \lim_X h = -\infty \}$, 
\item $(Sf)^{-1}(I)$ is finite for all bounded intervals $I$,
\end{enumerate} where $\lim_X h$ denotes $\lim_{x\in X, x\to \infty} h(x).$
\end{lemma} 

\begin{proof}
We first prove $(1)$. We aim to show $S^2 f = Sf$ for $f \in Lip(1,K)$. We have $S^2f=Sf=f$ on $K$. As $Sf$ only depends on $f|_K$, this shows that $S^2f=Sf$ proving $Sf=f$ for all $f \in \mathcal F$.

We next prove $(2)$. Let $f \in Lip(1,K)$.
Let $k,k' \in K,x,x'\in X$ and $y \in Y$. Then,
\begin{align*}
Sf(k)-Sf(k') &= f(k)-f(k') \leq d(k,k'),\\
Sf(x)-Sf(k)& = \sup_{w\in K} f(w) - f(k) - d(x,w) \leq \sup_{w\in K} d(k,w) - d(x,w) \leq d(x,k),\\
Sf(k)-Sf(x) &= f(k) -\sup_{w\in K} (f(w) -  d(x,w)) \leq f(k)-f(k) + d(x,k) \leq d(x,k),\\
Sf(y)-Sf(x) &\leq \inf_{w \in K} f(w) + d(y,w) - f(w) + d(x,w) = \inf_{w \in K}  d(y,w) + d(x,w)  =d(x,y),\\
Sf(x)-Sf(y) &= \sup_{w_1,w_2 \in K} f(w_1)-f(w_2) -d(x,w_1) - d(y,w_2) \\&\leq \sup_{w_1,w_2 \in K} d(w_1,w_2) -d(x,w_1) - d(y,w_2) \leq d(x,y),\quad \mathrm{and}\\
Sf(x)-Sf(x') &\leq \sup_w  f(w) -d(x,w) - f(w) + d(x',w) \leq d(x,x').
\end{align*}
The remaining cases can be proven analogously. This proves $Sf \in Lip(1)$, and thus, $\mathcal F \subseteq Lip(1)$.

The proof of $(3)$ is easy as supremum and infimum are monotonous.

Assertions $(4)$ and $(5)$ also follow easily as the supremum and infimum in the definition of $Sf$ are attained due to finiteness of $K$.
We next prove $(6)$. As $Sf \in Lip(1)$, we have $$Sf \geq \inf\{g \in Lip(1):g|_K = f|_K  \}.$$
Moreover if $g \in Lip(1)$ with $g|_K=f|_K$, then $g(x) \geq g(v) - d(x,v)$ for all $v \in K$. Taking supremum over $v$ shows $g \geq Sf$ on $X \cup K$. Taking the infimum, we have $$Sf \leq \inf\{g \in Lip(1):g|_K = f|_K  \}\quad \mathrm{on}\quad X\cup K.$$ This proves $(6)$.
The proof of $(7)$ works similarly.

We next prove $(8)$. We first show ``$\subseteq$".
We have already proven $\mathcal F \subset Lip(1)$.
We  next show $\nabla_+ Sf =1$ on $X$.
For $x \in X$, we have
\begin{align*}
1\geq \nabla_+ Sf(x) &= \sup_{y\sim x,w\in K} \frac{(Sf(w) - d(w,y)) - \sup_{v \in K} (Sf(v) - d(v,x) )}{d(x,y)}\geq \inf_{v \in K}\sup_{y\sim x}   \frac{(d(v,x) - d(v,y) )}{d(x,y)}  =1,
\end{align*}
where the second inequality follows by setting $w:=v$, and the last identity holds as $d$ is a path metric.
Similarly, one has $\nabla_- Sf = 1$ on $Y$.
%As $K$ is finite, we easily get $\sup_X Sf < \infty$ and $\inf_Y Sf >-\infty$.
We now show $\lim_X Sf = - \infty$.
As $Sf(x) \leq -d(x,K) + \sup_K f $ for $x \in X$, and as $K$ is finite and $X$ infinite, we get $\lim_X Sf = - \infty$. Similarly, we get $\lim_Y Sf = \infty$. 
 This shows ``$\subseteq$".

We next show ``$\supseteq$". Suppose that $$f \in \{h \in Lip(1) : (\nabla_+h)|_X =1, \;  (\nabla_-h)|_Y =  1, \; \lim_Y h = \infty ,\; \lim_X h = -\infty \}.$$
We aim to show $Sf = f$.
Let $x \in X$. By $(6)$, we have $Sf(x) \leq f(x)$.
We now show the reverse inequality. Let $(x=x_0,\ldots)$ be a maximal sequence in $X$ with $x_{n+1} \sim x_n$ and $f(x_{n+1}) = f(x_n) + d(x_n,x_{n+1})$ for all $n$. 
We note that $f(x)-f(x_n)=d(x,x_n)$ as $f \in Lip(1)$.
Since $\sup_X f < \infty$ and as balls are finite, the sequence has to be finite. Let $x_n$ be the last element. As $\nabla_+ f = 1$ on $X$, there exists $x_{n+1}\sim x_n$ with $f(x_{n+1}) = f(x_n)+d(x_n,x_{n+1})$. As $x_n$ is the last element, we infer $x_{n+1} \in K$.
Thus, 
\[f(x) = f(x_{n+1}) - d(x_n,x_{n+1}) - d(x,x_n) \leq f(x_{n+1}) - d(x,x_{n+1}) \leq Sf(x).\] This shows $Sf = f$ on $X$. Similarly, we get $Sf=f$ on $Y$, and hence $Sf = f$ implying $f \in \mathcal F$. This shows ``$\supseteq$" and proves $(8)$.

We next prove $(9)$. Let $g=Sf$. We write $R := \sup_K |g|$ and notice $g^{-1}([-r,r]) \subset  B_{R+r}(K)$ for all $r \in \R$ as 
\[
|Sf(x)| \geq \inf_{w \in K} d(x,w) - |f(w)| \geq d(x,K) - R .\] 
As $B_{r+R}(K)$ is finite, this shows that $(Sf)^{-1}(I)$ is finite for all bounded intervals $I$ as desired.

This finishes the proof of the lemma.
\end{proof}
We now give a variational principle to show that $\mathcal F$ contains functions with constant Laplacian on $K$.
This is the key ingredient for the proof of Theorem~\ref{thm:Main}.

\begin{lemma}\label{lem:constLaplace}
Let $G=(V,w,m,d)$ be a salami with connected salami partition $P=(X,Y,K)$.
Then,

\begin{enumerate}[(1)]
\item There exists a minimizer $f \in \mathcal F$ of $\max_K \Delta f$,
\item Every such minimizer $f$ satisfies $\Delta f =const$ on $K$.
\end{enumerate}

\end{lemma}
We remark that a dual version of the lemma holds with interchanging minima with maxima.

\begin{proof}
We first prove $(1)$. Let $x \in K$.
Let $\mathcal F_0 := \{f \in \mathcal F: f(x)=0\}$ equipped with the distance $$\rho(f,g) := \|f-g\|_\infty = \|(f|_K - g|_K)\|_\infty.$$
As $K$ is finite, we see that $\mathcal F_0$ is compact. As $\max_K \Delta f$ is continuous w.r.t. $\rho$, there must be a minimizer $f \in \mathcal F_0$ of $\max_K \Delta f$. This must also be a minimizer within $\mathcal F$ as $\Delta$ and $\mathcal F$ are invariant under adding constants. This proves $(1)$.

We now prove $(2)$.
Let $f$ be such a minimizer additionally minimizing $$|M_f|:=|\{x \in K:\Delta f(x) = \max_K \Delta f\}|.$$
Clearly, $|M_f|\geq 1$. If $|M_f|=|K|$, then $\Delta f = const$ on $K$ and there is nothing to prove. We finally assume $0<|M_f|<|K|$ and aim to find a contradiction.

We consider $g:=S (f + \eps \Delta f)$. By Lemma~\ref{lem:DiscreteTimeLipContraction}, we see that $ f + \eps \Delta f \in Lip(1,K)$ for small enough $\eps$ by non-negative curvature and finiteness of $K$.
Let $\delta = \max_K \Delta f - \max_{K\setminus M_f} \Delta f$.
If $\eps$ is small enough, then, $|\Delta f - \Delta g| < \delta$ on $K$. This gives  $\Delta g < \max_K \Delta f$ on $K \setminus M_f $.  Note that $f$ is a minimizer, and thus,  $\max_K \Delta f\leq \max_K\Delta g,$ which implies that $M_g \subseteq M_f.$

Let $x \in M_f$. Then, $g -g(x) \leq f-f(x)$ 
on $K$ as $g=f + \eps \Delta f$ on $K$. By monotonicity of $S$, we infer $g-g(x) \leq f-f(x)$ everywhere.
Hence, $\Delta g(x) \leq \Delta f(x)$
implying $\max_K \Delta g \leq \max_K \Delta f$.

As $0<|M_f|<|K|$ and by connectedness, there exist $x\sim y \in K$ with $\Delta f(x) = \max_K \Delta f> \Delta f(y)$. Hence, $$g(y)-g(x) = (f(y)+ \eps\Delta f(y)) - (f(x) + \eps\Delta f(x))<f(y)-f(x).$$ 
Together with $g-g(x) \leq f-f(x)$, this gives
 $\Delta g(x)<\Delta f(x)$ and thus $M_g \neq M_f$.
As $M_g \subseteq M_f$, this shows $|M_g| < |M_f|$ contradicting the minimality of $|M_f|$. This finishes the proof.
\end{proof}

\subsection{Existence of Lipschitz harmonic functions}

Having the variational principle and infinite measure of the ends, we can show that there are functions in $\mathcal F$ which are harmonic on $K$. 
\begin{lemma}{\label{lem:ZeroLaplace}}
Let $G=(V,w,m,d)$ be a salami with connected salami partition $P=(X,Y,K)$.
Let $f \in \mathcal F$ with $\Delta f=c=const$ on $K$. Then, $c=0$.
\end{lemma}
\begin{proof}

By non-negative Ollivier curvature and by Lemma~\ref{lem:Sproperties}, we have $\Delta f \geq c$ on $X$ and $\Delta f \leq c$ on $Y$.
We write $R:=\sup_K |f|$.
%We notice that if $|f(x)|\leq r$, then $d(x,K) \leq r +R$.
%and if $|f(x)|\geq r$, then $d(x,K) \geq r - R$.
%Particularly, the level sets of $f$ of bounded intervals are finite.
Let $X(n) := f^{-1}((-n, \infty)) \cap X  \subseteq f^{-1}((-n,R])$. 
We notice $|X(n)| < \infty$ by Lemma~\ref{lem:Sproperties}(9).
%Moreover, $X \subset f^{-1}((-\infty,R])$, and by finiteness of $f^{-1}([-R,R])$, we obtain  $m(X \cap f^{-1}((-\infty,R))) = \infty$.
As $m(X)=\infty$, we have
 $m(X(n)) \stackrel{n \to \infty}{\longrightarrow} \infty$.
 We assume $n > R$.
Now, we estimate
\[
c \cdot m(X(n)) \leq \langle \Delta f, 1_{X(n)} \rangle
=\left(\sum_{{\substack{y \in K\\ x \in X(n)}}}  + \sum_{{\substack{f(y)\leq  -n\\ x \in X(n)}}} \right) (f(y)-f(x))w(x,y)
\leq \sum_{{\substack{y \in K\\ x \in X }}} (f(y)-f(x))w(x,y)<\infty
\]
where the first inequality follows from $\Delta f \geq c$ on $X$, and the second inequality follows from non-positivity  of the second sum,
and finiteness follows as $K$ is finite.
Since $m(X(n))$ goes to infinity for large $n$, we can conclude $c \leq 0$.
Analogous considerations for $Y(n):=f^{-1}((-\infty,n)) \cap Y$ show $c \geq 0$ implying $c=0$.
\end{proof}

Combining the two lemmas shows that the Laplacian of functions in $\mathcal F$ cannot be strictly positive on $K$, as stated in the next lemma.

\begin{lemma}{\label{lem:NonPositiveLaplace}}
Let $G=(V,w,m,d)$ be a salami with connected salami partition $P=(X,Y,K)$.
Then,
\begin{enumerate}[(1)]
\item There is $f \in \mathcal F$ with $\Delta f = 0$ on $K$.
\item If $f \in \mathcal F$ with $\Delta f \leq 0$ on $K$. Then, $\Delta f=0$ on $K$. 
\item If $f \in \mathcal F$ with $\Delta f \geq 0$ on $K$. Then, $\Delta f=0$ on $K$. 
\end{enumerate}

\end{lemma}
%We remark that there is also a dual version of the lemma when replacing "$\geq$" by "$\leq$".

\begin{proof}
We first prove $(1)$. By Lemma~\ref{lem:constLaplace}, there exists $f \in \mathcal F$ with $\Delta f = c = const$ on $K$, and by Lemma~\ref{lem:ZeroLaplace}, we have $c=0$ proving $(1)$.

We next prove $(2)$.
Let $c := \min_{h\in\mathcal F} \max_K \Delta h$. By Lemma~\ref{lem:constLaplace}, there is $g \in \mathcal F$ with $\Delta g = c$. By Lemma~\ref{lem:ZeroLaplace}, we infer $c=0$. As $\Delta f \leq 0$ on $K$, we have $0 \geq \max_K \Delta f \geq c=0$.
Hence, $f$ is a minimizer of $\max_K \Delta h$ implying $\Delta f = c=0$ on $K$ by Lemma~\ref{lem:constLaplace}(2).
This proves $(2)$.

The proof of $(3)$ is similar. This finishes the proof.
\end{proof}

\subsection{Varying the salami partition}

In this section, we investigate how changing the connected salami partition affects the function spaces $\mathcal F$ and $\mathcal H$.

\begin{definition}
If $P=(X,Y,K)$ is a connected salami partition, and $\widetilde K\supseteq K$ finite and connected, we write $P(\widetilde K) := (X\setminus \widetilde K,Y \setminus \widetilde K,\widetilde K)$.
\end{definition}

We now show that $\mathcal F$ increases when $K$ increases.
\begin{lemma}\label{lem:superpartition}
Let $G=(V,w,m,d)$ be a graph with a salami partition $P=(X,Y,K)$. Let $\widetilde K\supseteq K$ be connected and finite. Then,
\begin{enumerate}[(1)]
\item $P(\widetilde K)$ is a connected salami partition
\item $\mathcal F(P) \subseteq \mathcal F(P(\widetilde K)).$
\end{enumerate}

\end{lemma}
\begin{proof}
The proof of $(1)$ is easy.
Assertion $(2)$ is an easy consequence of Lemma~\ref{lem:Sproperties}(8). This finishes the proof.
\end{proof}

Now, we increase $K$ to show that $\Delta f= 0$ everywhere, assuming only that $f \in \mathcal F$ and $\Delta f=0$ on $K$.

\begin{lemma}\label{lem:harmonicEverywhere}
Let $G=(V,w,m,d)$ be a salami with connected salami partition $P=(X,Y,K)$.
Suppose $\Delta f=0$ on $K$ for some $f \in \mathcal F$. Then, $\Delta f = 0$ everywhere.
In particular, $\mathcal H(P) \neq \emptyset.$
\end{lemma}

\begin{proof}
First let $x \in X$. We aim to show $\Delta f(x)=0$. We can choose $\widetilde K$ finite, connected such that $K \cup \{x\}\subset \widetilde K \subset K \cup X$.
By non-negative curvature, we have $\Delta f \geq 0$ on $X\cup K$ and thus on $\widetilde K$.
By Lemma~\ref{lem:superpartition}, we have $f \in \mathcal F(P(\widetilde K))$ and
 $ P(\widetilde K)$ is a connected salami partition.
Now, Lemma~\ref{lem:NonPositiveLaplace} shows $\Delta f = 0$ on $\widetilde K$ implying $\Delta f(x)=0$. Similarly, one can show $\Delta f=0$ on $Y$. By Lemma~\ref{lem:NonPositiveLaplace}, we have $\mathcal H(P) \neq \emptyset$. This finishes the proof.
\end{proof}

%The lemma particularly shows that $\mathcal H(P) \neq \emptyset$ for connected salami partitions $P$.
We can now deduce recurrence of salamis.
We recall a graph $G=(V,w,m)$ is called \emph{recurrent} if for all $\eps>0$, there exists a function $f \in C_c(V)$ with $\|f\|_\infty = 1$ and $-\langle f,\Delta f \rangle <\eps$.
We first give a recurrence criterion for general graphs.

\begin{lemma}\label{lem:recurrence}
Let $G=(V,w,m,d)$ be a locally finite connected graph. Suppose there exists a harmonic function $f \in Lip(1)$ with $|f^{-1}(I)|< \infty$ for all bounded intervals $I$. 
Moreover, assume 
\[|E(f^{-1}((-\infty,r)),f^{-1}([r,\infty))) |< \infty \mbox{ for all }r \in \R.\]
Then, $G$ is recurrent.
\end{lemma}
\begin{proof}
W.l.o.g., we assume that $f^{-1}(\{0\})\neq\emptyset.$
Let $R>0$.
We consider $g := (R - |f|)_+$. We notice $g \in C_c(V)$ as $|f^{-1}([-R,R])|< \infty$.
Set 
\[
M:=\{x \in V: \sgn f(x) \neq \sgn f(y) \mbox{ for some } y \sim x\}.
\] Note that $\Delta g \geq 0$ on $V\setminus M$ as $\Delta f =0$.
Moreover, $M$ is finite as $E(f^{-1}((-\infty,r)),f^{-1}([r,\infty))) < \infty$ and as $f^{-1}(\{0\})$ is finite.
Hence with $\Deg_d(x) := \Delta d(x,\cdot)(x)$,
\[
-\langle g, \Delta g \rangle \leq \langle g 1_M, -\Delta g \rangle \leq  \langle g 1_M, \Deg_d \rangle \leq R\langle 1_M,\Deg_d \rangle =CR < \infty  
\]
where the second inequality follows as $g \in Lip(1)$, and finiteness follows as $M$ is finite.
Moreover, $\|g\|_\infty =R$ as $f^{-1}(\{0\}) \neq \emptyset$.
Thus,
\[
-\frac{\langle g,\Delta g\rangle}{\|g\|_\infty^2} \leq \frac{CR}{R^2} \stackrel{R\to \infty}{\longrightarrow} 0.
\]
Normalizing $g$ implies recurrence. This finishes the proof.
\end{proof}

Using the lemma, we now show recurrence of salamis.

\begin{theorem}\label{thm:recurrent}
Every salami is recurrent.
\end{theorem}
\begin{proof}
Let $G=(V,w,m,d)$ be a salami.
Let $f \in \mathcal H(P)$ for some connected salami partition $P$.
 Lemma~\ref{lem:Sproperties}  shows that 
$f$ has finite level sets. 
Let $r \in \R$ and
let $A:=f^{-1}((-\infty,r))$ and $B:=f^{-1}([r,\infty))$.
For applying Lemma~\ref{lem:recurrence}, 
we have to show that $E(A,B) < \infty$.

Let $R>0$ such that $-R \leq f|_K \leq R$ and such that $R>|r|$. Then $f|_X \leq R$ and $f|_Y \geq - R$. Hence, there are no edges from $f^{-1}((-\infty,-R)) \subseteq X$ to $f^{-1}((R,\infty)) \subseteq Y$.
As $R>|r|$, every edge from $(x,y) \in E(A,B)$ must contain a vertex of $f^{-1}([-R,R])$ which is finite. Thus, $ E(A,B)$ is finite so we can apply Lemma~\ref{lem:recurrence}.
Hence, $G$ is recurrent. This finishes the proof.
\end{proof}

We next show that $\mathcal H(P)$ has only one element up to adding constants.

\begin{lemma}\label{lem:Honedimensional}
Let $G=(V,w,m,d)$ be a salami with connected salami partition $P=(X,Y,K)$.
Then,
\[
\mathcal H(P) = h+ \R \mbox{ for some } h \in \mathcal H(P).
\]
\end{lemma}

\begin{proof}
Let $h_1,h_2 \in \mathcal H(P)$. By adding a constant to $h_2$, we can assume $h_1 \leq h_2$ on $K$ and $h_1(x)=h_2(x)$ for some $x \in K$.
By Lemma~\ref{lem:Sproperties}, we infer $h_1 \leq h_2$ everywhere.

Let $M:=\{v \in K: h_1(v)=h_2(v)\} \neq \emptyset$. We now show $M=K$. Suppose not. Then, by connectedness of $K$, there exist $v\sim w$ with $v \in M$ and $w \in K\setminus M$. Hence, $\Delta h_1(v) < \Delta h_2(v)$ contradicting $\Delta h_1 = \Delta h_2=0$. Thus, $h_1=h_2$ on $K$. Therefore,
$h_1 = Sh_1 = Sh_2 = h_2$. This shows $H(P) = h_1+ \R$ and finishes the proof.
\end{proof}

We now show that $\mathcal H(P)$ stays invariant under finite variations of $P$.

\begin{lemma}\label{lem:HsameFinite}
Let $P_i=(X_i,Y_i,K_i)$ for $i=1,2$ be connected salami partitions with $X_i \setminus X_j$ and $Y_i \setminus Y_j$ finite for $i,j=1,2$.
Then, $\mathcal H(P_1)= \mathcal H(P_2)$.
\end{lemma}

\begin{proof}
Let $K$ be connected, finite and contain $K_i$ and $X_i \setminus X_j$ and $Y_i \setminus Y_j$ for $i,j=1,2$. Then, $P_1(K) = P_2(K)$.

By Lemma~\ref{lem:Honedimensional} and Lemma~\ref{lem:superpartition}, there exist $h$ and $h_i$ for $i=1,2$ such that 
\[
h_i+\R=
\mathcal H(P_i) \subseteq \mathcal H(P_i(K)) = h + \R,\quad i=1,2,
\]
implying $\mathcal H(P_1)=\mathcal H(P_1(K)) = \mathcal H(P_2(K)) = \mathcal H(P_2)$. This finishes the proof.
\end{proof}

We now show that that $\mathcal H(P)$ stays invariant when increasing $X$ and fixing $K$.

\begin{lemma}\label{lem:SameH}
Let $G=(V,w,m,d)$ be a salami and $P_i=(X_i,Y_i,K)$, $i=1,2$ be connected salami partitions with $X_1 \subseteq X_2$. Then, $\mathcal H(P_1)=\mathcal H(P_2)$.
\end{lemma}

\begin{proof}
Let $Z:= X_2 \setminus X_1 $.
Let $f \in \mathcal H(P_1)$.
Let $g:=S(P_2) f \in \mathcal F(P_2)$.
Then, $g = f$ on $V \setminus Z$ and $g \leq f $ on $Z$ by Lemma~\ref{lem:Sproperties}(6).
Suppose $g(x)<f(x)$ for some $x \in Z$. Then, there exists  $R \geq 0$ such that $f = g$ on $\widetilde K := B_R(K)$ and $f(y) \neq g(y)$ for some $y \in V$ adjacent to $\widetilde K$ .
Let $\widetilde P_2 = P_2(\widetilde K)$. 
%$\widetilde P_2 = (X_2 \setminus \widetilde K, Y_2 \setminus \widetilde K,\widetilde K) $
By Lemma~\ref{lem:superpartition}, we have $g \in \mathcal F(\widetilde P_2)$. Moreover $\Delta g \leq 0$ on $\widetilde K$ and $\Delta g(x)<0$ for some $x \in \widetilde K$ 
as $g(y)<g(y)$ for some $y$ adjacent to $\widetilde K$.
 This is a contradiction to Lemma~\ref{lem:NonPositiveLaplace}(2) proving $f=g$.
Thus, $f \in \mathcal H(P_2)$ showing $\mathcal H(P_1) \subseteq \mathcal H(P_2)$.
The inclusion $\mathcal H(P_1) \supseteq \mathcal H(P_2)$ follows similarly. This finishes the proof.
\end{proof}

We will now prove the first part of the main theorem.

\begin{theorem}\label{thm:ends}
Every salami has exactly two ends.
\end{theorem}
\begin{proof}
Suppose there is a salami $G=(V,w,m,d)$ with three or more ends. Then, there exist $X,Y,Z,K \subset V$ such that
\begin{enumerate}[(a)]
\item $V = X \dot \cup Y \dot \cup Z \dot \cup K$,
\item $m(X)=m(Y)=\infty$,
\item $|Z|=\infty$,
\item $K$ is finite and connected,
\item $E(X,Y)=E(Y,Z)=E(Z,X)=\emptyset$.
\end{enumerate}
Let $P_1 = (X ,Y \cup Z,K)$ and $P_2 = (X\cup Z,Y,K)$. Then, $P_1$ and $P_2$ are connected salami partitions and by Lemma~\ref{lem:SameH}, we have $\mathcal H(P_2) = \mathcal H(P_1) \neq \emptyset$. This is a contradiction as by Lemma~\ref{lem:Sproperties}(8),  every function in $\mathcal F(P_1)$ goes to $\infty$ along $Z$ and every function in $\mathcal F(P_2)$ goes to $-\infty$ along $Z$. This finishes the proof.
\end{proof}

Using that every salami has exactly two ends, we now show that there are harmonic functions with constant gradient.

\begin{theorem}\label{thm:constGradient}
Let $G=(V,w,m,d)$ be a salami with connected salami partition $P=(X,Y,K)$. Let $f \in \mathcal H(P)$. Then,
\begin{enumerate}[(a)]
\item $\Delta f = 0$,
\item $\nabla_+ f = \nabla_- f = 1$.
\end{enumerate}
\end{theorem}

\begin{proof}
It is clear that $\Delta f = 0$.
We now show $\nabla_- f = 1$. Let $R \geq 0$.
Let $\widetilde Y = Y \cup B_R(K)$.
By Theorem~\ref{thm:ends}, $V\setminus \widetilde Y$ has only one infinite connected component $X_0$, and maybe some finite connected components collected in $Y_0$.
Let $K_0 := \{x \in X_0:x \sim y \mbox{ for some } y \in \widetilde Y \}$. We see that $K_0$ is finite as every $x \in K_0$ is adjacent to some $y \in B_R(K)$.
Hence, $(X_0 \setminus K_0, \widetilde Y \cup Y_0,K_0)$ is a (not necessarily connected) salami partition.

As $X_0$ is connected, there exists $K_1$ finite and connected with $K_0 \subseteq K_1 \subseteq X_0$. Thus,
$\widetilde P=(X_0\setminus K_1, \widetilde Y \cup Y_0, K_1)$ is a connected salami partition.
By Lemma~\ref{lem:HsameFinite}, we get $f \in \mathcal H(\widetilde P)$ and thus by Lemma~\ref{lem:Sproperties}(8), we have $\nabla_-f= 1$ on $\widetilde Y \cup Y_0 \supseteq B_R(K)$.
Taking $R\to \infty$ shows $\nabla_- f=1$ everywhere.
Similarly, one can show $\nabla_+f = 1$. This finishes the proof.
\end{proof}

Employing the harmonic functions with constant gradient, we can finally prove flatness.

\begin{corollary}\label{cor:flat}
Every salami is flat.
\end{corollary}
\begin{proof}
By Theorem~\ref{thm:constGradient}, there exists $f \in Lip(1)$ with $\nabla_{\pm} f = 1$ and $\Delta f=0$.  
Let $x \in V$ and $y \sim x$ with $f(y)-f(x)=d(x,y)$.
Then,
\[
\kappa(x,y) = \inf_{\substack{g \in Lip(1)\\g(y)-g(x)=d(x,y)}} \Delta g(x)-\Delta g(y) \leq \Delta f(x) - \Delta f(y) = 0
\]
showing $\kappa(x) \leq 0$. As every salami has non-negative Ollivier curvature, we conclude $\kappa(x)=0$. Since $x$ is arbitrary, this shows that every salami is flat.
\end{proof}
In fact, by the proof of the above corollary we have a stronger result:
for any $x\in V,$ there are distinct $y_i \sim x,$ $i=1,2,$ such that $\kappa(x,y_1)=\kappa(x,y_2)=0.$ 

\section{Level sets of Lipschitz harmonic functions} \label{sec:levelSets}

In this section, we show that the level sets of functions in $\mathcal H$ are connected and uniformly bounded in size. By this, we show that salamis are quasi-isometric to the line. 
From now on, we restrict ourselves to using the combinatorial distance instead of a path metric.

\subsection{Connected level sets}

We now show that the level sets of $f \in \mathcal H$ are connected.
\begin{lemma}\label{lem:levelsetConnected}
 Let $G=(V,w,m,d_0)$ be a salami, $P$ a connected salami partition, and $f \in \mathcal H(P)$. Then, $ f^{-1}([a,b))$ is connected whenever $b-a\geq 2$.

\end{lemma}
\begin{proof}
We first prove that
$M_a := f^{-1}((-\infty, a))$ is connected.
Suppose $M_a$ has a finite connected component $F$. Let the minimum of $f|_F$ be attained in $x$. As $\nabla_- f(x) = 1$, there is $y \sim x$ with $f(y)< f(x)$. By minimality, $y \notin F$. As $F$ is a connected component, $y \notin M_a \setminus F$. Hence, $y \in V\setminus M_a$ implying $f(y) \geq a$ contradicting $f(y) < f(x)<a$. This shows that $M_a$ has no finite connected component.
However if $M_a$ has at least two infinite connected components, then $V \setminus f^{-1}([a,a+1))$ has at least three infinite connected components contradicting Theorem~\ref{thm:ends}. This proves the claim that $M_a$ is connected. By a similar argument, $M^b := f^{-1}([b,\infty))$ is also connected.
As the level sets are finite, we also see $m(M_a)=m(M^b)= \infty$.

We now suppose, $K := f^{-1}([a,b))$ is disconnected. Let $K_1$ be a connected component of $K$, and $K_2 := K \setminus K_1$. We observe 
\[
E(K_i,M_a) \neq \emptyset \neq E(K_i,M^b)
\]
as for every $x \in K_i$ there exist paths $(\ldots x_{-1},x_0=x,x_1 \ldots)$ with $f(x_i)-f(x_j)=i-j$ and thus, with appropriate choice of $n_a, n_b \in \Z$, one has $x_n \in M_a$ for $n<n_a$ and $x_n$ in $K_i$ for $n_a \leq n <n_b$ and $x_n \in  M_b$ for $n \geq n_b$.
Particularly, $V\setminus K_i$ is connected for $i=1,2$. Hence, $G$ stays connected when removing all edges from $M_a$ to $K_1$.

We now construct a covering graph of $G$ given by $\widetilde G=( V \times \Z,\widetilde w,\widetilde m)$ with $\widetilde m((x,n)) = m(x)$ and
\[
\widetilde w((x,n),(y,m)) = \begin{cases}
w(x,y)&: m=n \mbox{ and } (x,y) \notin E(M_a,K_1)\\
w(x,y)&: m=n+1 \mbox{ and } (x,y) \in E(M_a,K_1)\\
0&: \mbox{otherwise}.
\end{cases}
\]
We now claim that $\widetilde G$ is a salami.
Let $G'$ be the graph emerging from $G$ by setting $w'(x,y)=w'(y,x)=0$ for all $(x,y) \in E(M_a,K_1)$.
We claim that for $(x,y) \in E(M_a,K_1)$, every path from $x$ to $y$ within $G'$ has length at least $5$.
Let $(x=x_0,\ldots x_n=y)$ be a path from $x$ to $y$ in $G'$.
We notice $x_n,x_{n-1} \in K_1$ as $f(x_n) < a+1$ and $f(x_{n-1})<a+2 \leq b$ as $f \in Lip(1)$. By the same argument, the path has to contain at least two vertices in $K_2$ as the only way to exit $M_a$ within $G'$ is through $K_2$.
Moreover, the path has to contain a vertex in $M^b$ as 
the only way to exit $K_1$ within $G'$ is through $M_b$. Thus, the path must contain at least six vertices showing that its length is at least five.

By this, one can show that $\widetilde G$ is a covering of $G$ which preserves all cycles of length five or smaller. As the Ollivier curvature only depends on the cycles of length five, this shows that $G'$ also preserves the Ollivier curvature, meaning that $\widetilde G$ also has non-negative Ollivier curvature.

We now show that $\widetilde G$ has three ends with infinite measure.
This follows as $(V\times \Z) \setminus (K \times \{0\})$ has three connected components of infinite measure within $\widetilde G$, namely $M^b \times \{0\}$, and $(V \times \N_+) \cup (M^a \times \{0\})$, and $V \times (-\N_+)$. Hence, $\widetilde G$ is a salami with three ends contradicting Theorem~\ref{thm:ends}. 
This shows the falsity of our assumption that $f^{-1}([a,b))$ is disconnected, and finishes the proof.
\end{proof}

\subsection{Quasi-isometry to the line}

We show that the level sets of functions in $\mathcal H$ have uniformly bounded cardinality when assuming lower bounded edge weights. The necessity of this assumption is demonstrated in Example~\ref{ex:edgeweightnotlowerbounded}.

\begin{lemma}\label{lem:levelsetSize}
Let $G=(V,w,m,d_0)$ be a salami with connected salami partition $P=(X,Y,K)$. Suppose there is $\eps>0$ such that $w(x,y) \geq \eps$ for all $(x,y) \in E$. Let $f \in \mathcal H(P)$. Then, for all $r \in \R$, 
\[|f^{-1}([r-1,r))| \leq \frac 1 \eps \sum_{\substack{f(y)<0\\f(x)\geq 0}}  w(x,y) (f(x)-f(y)) .\] 
\end{lemma}
\begin{proof}

W.l.o.g., $r>1$. We consider $M := f^{-1}([0,r))$. 
By Lemma~\ref{lem:Sproperties}(9), $M$ is finite and thus,
\begin{align*}
0 = \langle \Delta f, 1_M \rangle = \left( \sum_{\substack{f(y)\geq r \\f(x)<r}} + \sum_{\substack{f(y)<0\\f(x)\geq 0}} \right) w(x,y) (f(y)-f(x))  =: s_r - s_0
\end{align*}
where the second sum $-s_0$ is negative and independent of $r$.
We write $Q(r):= f^{-1}([r-1,r))$.
For the first sum, we notice that for every $x \in Q(r)$ there is $y \sim x$ with $f(y)-f(x)=1$, implying $f(y) \geq r$ and thus,
\begin{align*}
 s_r = \sum_{\substack{f(y)\geq r \\f(x)<r}}  w(x,y) (f(y)-f(x)) 
\geq \sum_{x \in Q(r)} \eps \cdot 1 
=\eps |Q(r)| 
\end{align*}
Therefore, $|Q(r)| \leq s_0/\eps$ which finishes the proof.
\end{proof}

Combining with Lemma~\ref{lem:levelsetConnected}, we can show that $f \in \mathcal H$ is a quasi-isometric embedding.

\begin{theorem}\label{thm:quasiIsom}
Let $G=(V,w,m,d_0)$ be a salami with connected salami partition $P=(X,Y,K)$. Suppose there is $\eps>0$ such that $w(x,y) \geq \eps$ for all $(x,y) \in E$. Let $f \in \mathcal H(P)$. Then, for all $x,y \in V$,

\[
|f(x) - f(y)| \leq d(x,y) \leq |f(x) - f(y)| +  C 
\]
with
\[
C=\frac 2 \eps \sum_{\substack{f(y)<0\\f(x)\geq 0}}  w(x,y) (f(x)-f(y)) .
\]
Particularly, $f$ is a quasi-isometric embedding into $\R$.
\end{theorem}

\begin{proof}
The first inequality is clear as $f \in Lip(1)$. 
For the second inequality, first assume $|f(y)-f(x)|<2$. W.l.o.g., $f(y)>f(x)$. 
For simplicity, we write $d=d_0$.
Let $M:=f^{-1}([f(x),f(x)+2))$. Then, $x,y \in M$ and $M$ is connected by Lemma~\ref{lem:levelsetConnected}, and $|M| \leq C$ by Lemma~\ref{lem:levelsetSize}. Thus, $d(x,y) \leq C$ proving the theorem for $x,y$ with $|f(y)-f(x)| \leq 2$.

For $|f(y)-f(x)|\geq 2$, we again assume $f(y)>f(x)$ without obstruction.
As $\nabla_-f=1$, there exists $z\in V$ with $f(z) \in [f(x),f(x+2))$ and $d(y,z)=f(y)-f(z)$. As we have already shown, we have $d(x,z) \leq C$. Therefore,
\[
d(x,y) \leq d(y,z) + d(x,z) \leq f(y)-f(z) + C \leq f(y)-f(x) + C
\]
This finishes the case distinction and the proof.
\end{proof}

\section{Harmonic functions on salamis}\label{sec:harmonic}

In this section, we investigate harmonic functions on salamis.
We prove that the dimension of the space of all harmonic functions is finite, and that the dimension of the space of harmonic functions of subexponential growth is two. Moreover, we give a Cheng-Yau gradient estimate for harmonic functions on balls.
In the whole section, we only consider the combinatorial distance $d_0$ instead of a general path distance.

\subsection{Uniqueness of Lipschitz sharp harmonic functions}
Let $G=(V,w,m,d_0)$ be a salami and let 
\[
\mathcal H_0 := \{f \in Lip(1):\Delta f = 0, \nabla_- f=\nabla_+ f=1\}.
\] 
We prove the uniqueness of harmonic functions in $\mathcal H_0.$

\begin{theorem}\label{thm:H0Characterization}
Let $G=(V,w,m,d_0)$ be a salami. Let $f,g \in \mathcal H_0$. Then,
\[
f-g = const \quad \mbox{ or } \quad  f+g = const.
\]
\end{theorem}

\begin{proof}
W.l.o.g., $f \in \mathcal H(P)$ for some connected salami partition $P$.
W.l.o.g., $f^{-1}(\{\Z + \frac 1 2 \}) = \emptyset$.

Let $A:=f^{-1}((-\frac 1 2 , \frac 1 2))$.
W.l.o.g., $\max_A g = 0$, and let the maximum be attained at $x_0 \in A$.
Let $y_0 \sim x_0$ with $g(y_0)=g(x_0)+1 = 1$.
Then, $y_0 \notin A$.
By changing sign of $f$, we can assume $f(y_0)>\frac 1 2$.

We write  $B:= f^{-1}((\frac 1 2,\frac 3 2])$ and $Y := f^{-1}((\frac 3 2,\infty))$, and $X := f^{-1}((-\infty,-\frac 1 2]) $.

We now claim that $g-f \leq \max_B (g-f)$ on $Y$.
For simplicity, we write $d=d_0$.
Let $y \in Y$. Then, there is $z \in B$ with $f(y)-f(z)=d(y,z)$. Moreover, $g(y)-g(z) \leq d(y,z)$ implying $$g(y)-f(y) \leq g(z) - f(z)\leq \max_B g-f,$$ which yields $g-f \leq \max_B g-f$ on $Y$.

We next claim that $\max_B g = 1$. Suppose $g(x)>1$ for some $x \in B$. Then, there is $y \sim x$ with $f(y)=f(x)-1$ implying $y \in A$. Hence, $g(y) \geq g(x) - 1 >0$ contradicting $\max_A g = 0$. As $g(y_0)=1$, this shows $\max_B g = 1$.

Let $C:=\max_B (g -f)$ be attained at some $x_1 \in B$.
We notice $C>-\frac 1 2$ as $g(y_0)=1$ and $f< \frac 3 2$ on $B$.
Thus, $g(x_1) = C + f(x_1)>0$ as $f>\frac 1 2$ on $B$. Let $y_1 \sim x_1$ with $g(y_1) - g(x_1) =1$. Then, $g(y_1)>1$ implying $y_1 \notin A\cup B$. As $y_1 \sim x_1 \in B$, we see $y_1 \notin X$ and hence, $y_1 \in Y$. As $f \in Lip(1)$, we obtain $$g(y_1) -f(y_1) \geq g(y_1) - f(x_1) - 1 = g(x_1)-f(x_1) = C.$$ As $g-f \leq C$ on $Y$, this shows $g(y_1)-f(y_1)=C$. By the maximum principle applied to $g-f$, and as $B\cup Y$ is connected, and as all neighbors of $y_1$ are in $B \cup Y$, this shows $g-f = C$ on $B \cup Y$. As $f$ is the minimal 1-Lipschitz extension on $A \cup X$ when fixing $f$ on $B \cup Y$, we see $f \leq g - C$ on $A \cup X$. As $\Delta (g-f) = 0$, we infer $g-f=C$ everywhere.
\end{proof}

\subsection{Finite dimension of the space of harmonic functions}

In the following, we prove that the space of all harmonic functions on a salami is of finite dimension. In particular, we need not assume any growth condition on harmonic functions. 
\begin{theorem}\label{thm:finiteDimention}
Let $G=(V,w,m,d_0)$ be a salami. Suppose there is $\eps>0$ such that $w(x,y) \geq \eps$ for all $(x,y) \in E$. 
Then the dimension of the space of harmonic functions on $G$ is bounded above by 
\[\frac 2 \eps \sum_{\substack{f(y)<0\\f(x)\geq 0}}  w(x,y) (f(x)-f(y)),
\]
where $f \in \mathcal H(P)$ for a connected salami partition $P.$
\end{theorem}
\begin{proof} By Lemma~\ref{lem:levelsetConnected}, 
\[|f^{-1}([r-1,r))| \leq C, \quad \forall r\in \R,\] where $C=\frac 1 \eps \sum_{\substack{f(y)<0\\f(x)\geq 0}}  w(x,y) (f(x)-f(y)).$ To prove the result, it suffices to show that for any finite dimensional linear space of harmonic functions $Q,$ $$\dim Q\leq 2C.$$

For any $n\in \N_+,$ set $\Omega_n:=f^{-1}([-n,n]).$ Then $\{\Omega_n\}_{n=1}^\infty$ is an exhaustion of $V,$ i.e.
$$|\Omega_n|<\infty,\quad \Omega_n\subset \Omega_{n+1}, \forall n\in \N_+,\quad V=\cup_{n=1}^\infty\Omega_n.$$ Denote by $\partial \Omega_n:= B_1(\Omega_n)\setminus \Omega_n$ the boundary of $\Omega_n$ in $V.$ Since $f\in Lip(1),$ $$\partial \Omega_n\subset f^{-1}([-n-1,-n)\cup (n,n+1]).$$ Hence, $$|\partial \Omega_n|\leq 2C.$$
For any $n\in \N_+,$ we define %\begin{eqnarray*}T_n:& Q\to \{f:\partial \Omega_n\to\R\}\\
%&f\mapsto f|_{\partial\Omega_n}.
%\end{eqnarray*}
\[\begin{array}{ll}T_n:& Q\to \{g:\partial \Omega_n\to\R\}\\
&f\mapsto f|_{\partial\Omega_n}.
\end{array}\] We claim that $T_n$ is injective for sufficiently large $n.$
Suppose it is not true, then there exists $\{g_n\}_{n=1}^\infty\subset Q\setminus\{0\}$ such that $T_n(g_n)=0,\forall n\in \N_+.$ Let $\dim Q=k$ and let $\{f_i\}_{i=1}^k$ be a basis of $Q.$ W.l.o.g., we assume that $$g_n=\sum_{i=1}^k a_i^n f_i,\quad \sum_{i=1}^k(a_i^n)^2=1.$$ By the compactness of $(k-1)$-dimensional unit sphere $\mathbb{S}^{k-1},$ there exists a subsequence, still denoted by $\{(a_1^n, \cdots, a_k^n)\}_n,$ and $(a_1^\infty, \cdots, a_k^\infty)\in \mathbb{S}^{k-1}$ such that
$$(a_1^n, \cdots, a_k^n)\to (a_1^\infty, \cdots, a_k^\infty),\quad n\to\infty.$$ Set $g_\infty=\sum_{i=1}^k a_i^\infty f_i.$ Note that $g_\infty\neq 0.$ Moreover, $g_n\to g_\infty$ pointwise as $n\to\infty$. Since $T_n(g_n)=g_n|_{\partial \Omega_n}=0,$ by the maximum principle of harmonic functions, $g_n|_{\Omega_n\cup \partial \Omega_n}=0.$ Since $\{\Omega_n\}_n$ is an exhaustion of $V,$ $g_n\to 0, n\to\infty.$ This yields a contradiction, and proves the claim.

Since $T_n$ is injective for sufficiently large $n,$
$$\dim Q= \dim T_n(Q)\leq |\partial \Omega_n|\leq 2C.$$ This finishes the proof.
\end{proof}

\subsection{Harmonic functions of subexponential growth}

In this subsection, we prove that all harmonic functions with subexponential growth are multiples of functions in $\mathcal H_0$.
Particularly, the dimension of the space of harmonic functions with subexponential growth is two.

We first prove that salamis have a global upper bound for the weight of edges $(x,y)$ with $f(x) \neq f(y)$ for some $f \in \mathcal H_0$.
In Example~\ref{ex:noGlobalEdgeWeightBound} it is shown that one cannot get an upper edge weight bound for all edges.
\begin{lemma}\label{lem:wUpperBound}
Let $G=(V,w,m,d_0)$ be a salami. Let $f \in \mathcal H_0$. Then, there exists $C \in \R$ such that $w(x,y) \leq C$ whenever $f(x) \neq f(y)$.
\end{lemma}
\begin{proof}
Let $P=(X,Y,K)$ be a salami partition. By Theorem~\ref{thm:H0Characterization}, we can assume $f \in \mathcal H(P)$.
As $f \in \mathcal F(P)$ and as we consider the combinatorial graph distance, we have $f(V) \subseteq f(K) + \Z$.
As $K$ is finite,  there is $\eps> 0$ such that $|f(y)-f(x)|> \eps$ whenever $f(y) \neq f(x)$.
Let $x_0\sim y_0$ with $f(y_0)>f(x_0)$. We write $A:=f^{-1}((-\infty, f(x_0)])$ and notice
\[
\langle \Delta 1_X, f \rangle = \langle \Delta 1_A,f \rangle = \sum_{\substack{x \in A \\ y \notin A}} (f(y) - f(x)) w(x,y)  \geq (f(y_0)-f(x_0))w(x_0,y_0) \geq \eps w(x_0,y_0).
\]
Hence, $w(x,y) \leq \frac{\langle \Delta 1_X, f \rangle}{\eps}$
 whenever $f(x) \neq f(y)$. This finishes the proof.
\end{proof}

We next show that certain harmonic functions on salamis have a sign change close to every vertex.
\begin{lemma}\label{lem:signBothSides}
Let $G=(V,w,m,d_0)$ be a salami with
 salami partition $P=(X,Y,K)$.
Let $f \in \mathcal H_0$.
Let $u$ be a harmonic function. Suppose that $\langle u , \Delta 1_X \rangle = 0$.
Let $A_r := f^{-1}((r-1,r])$.
Then there exists $c \in \R$ such that $\max_{A_r} u \geq c$ and $\min_{ A_r} u \leq c$ for all $r\in \R$.
\end{lemma}
\begin{proof}
Suppose not. Then, $\inf_r \max_{A_r} u < \sup_r \min_{A_r} u$.
Hence, there exist $r,R$ such that $\max_{A_r} u < c < \min_{A_R} u$.
W.l.o.g., $R>r$.
Let 
\[
A:= f^{-1}((-\infty,r-1]) \cup     \left(u^{-1}((-\infty,c] )
\cap f^{-1}((-\infty,R])\right)
\]
As $\langle u,\Delta 1_X \rangle = 0$ and as $u$ is harmonic, we also have $\langle u , \Delta 1_A \rangle = 0$.
We claim that for any $A \ni x \sim y \notin A$, we have $u(x)\leq c < u(y)$.
Suppose $u(x)>c$. Then, $x \in A_{r-1}$ and $y \in A_r$. However, $u<c$ on $A_r$ implying $y \in A$ which is a contradiction showing $u(x) \leq c$. Similarly we obtain $u(y)>c$.
Thus,
\[
\langle u,\Delta 1_A \rangle = \sum_{\substack{x\in A \\ y \notin A}} (u(y)-u(x))w(x,y) >0
\]
This is a contradiction proving the lemma.
\end{proof}

Assuming a lower bound on the edge weight, we finally prove that harmonic functions with subexponential growth have to be constant or to coincide with a multiple of a function in $\mathcal H_0$.

\begin{theorem}
Let $G=(V,w,m,d_0)$ be a salami. Suppose $0< \eps< w(x,y)$ for all $x \sim y$. Let $u:V\to \R$ be harmonic and $x_0 \in V$. Assume 
\[
\max_{B_r(x_0)} |u| = e^{o(r)}, \quad \mathrm{as}\ r\to\infty, 
\]
Then, $u \in \R\mathcal H_0 + const$.

\end{theorem}

\begin{proof}
Let $P=(X,Y,K)$ be a salami partition.
By subtracting a function in $\R \mathcal H_0$ from $u$ we can assume that
$\langle u, \Delta 1_X \rangle = 0$.
We write $C_r = f^{-1}((-\infty,2r])$ and
$D_r = C_{r+1} \setminus C_r$. 
 By Lemma~\ref{lem:signBothSides}, we can  assume $\max_{D_r} u \geq 0 \geq \min_{D_r} u$  for all $r$ by subtracting a constant from $u$.

We will now show that $u=const.$ 
We have $|D_r| \leq c$ for some $c \in \R$ independent of $r$ by Lemma~\ref{lem:levelsetSize}. 
By Lemma~\ref{lem:wUpperBound} there exists $C$ such that $w(x,y)< C$ for all $(x,y) \in E(C_r, V\setminus C_r)$.
Thus,
\[
\alpha_r := \langle \Delta 1_{C_r}, u^2 \rangle= \sum_{x \in C_r,y \notin C_r} (u^2(y)-u^2(x))w(x,y)  \leq Cc^2 \max_{y \in D_r} u^2(y).
\]
Let $y \in D_r$ maximizing $|u(y)|$.
As $\max_{D_r} u \geq 0 \geq \min_{D_r} u$, there is $y' \in D_r$ with $u(y)u(y') \leq 0$. As $|D_r|<c$ and as $D_r$ is connected by Lemma~\ref{lem:levelsetConnected}, there exists $z \in D_r$ with $c |\nabla u(z)| \geq \max_{D_r} |u|$.
As $w(x,y) \geq \eps$ for all $x\sim y$, we have
$\Delta u^2  \geq \frac \eps {m} |\nabla u|^2$.
Therefore,
\[
\alpha_{r+1} - \alpha_r =  \langle \Delta 1_{D_r},u^2\rangle
= \langle 1_{D_r}, \Delta  u^2\rangle
 \geq \eps \sum_{x \in D_r} |\nabla u(x)|^2 \geq \eps |\nabla u(z)|^2   \geq \frac \eps {c^2} \max_{D_r} u^2 \geq \frac \eps {Cc^4} \alpha_r.
\]
Thus, $\alpha_r$ is growing at least exponentially in $r$ if $\alpha_r>0$. This however is impossible as $u$ has subexponential growth, and as the number of edges out of $C_r$ is uniformly bounded.
This shows that $\alpha_r \leq 0$ for all $r \in \Z$.
When considering $C'_r = f^{-1}((2r, \infty))$ instead of $C_r$, we get $\alpha_r \geq 0$ by a similar argument.
Hence, $\alpha_r = 0$ implying $\eps \sum_{D_r} |\nabla u(x)|^2 \leq  \alpha_{r+1} - \alpha_r = 0$ for all $r \in \Z$. Thus, $u$ is constant which finishes the proof.
\end{proof}

By Theorem~\ref{thm:H0Characterization}, the dimension of $\R \mathcal H_0 +const$ is two, and thus we obtain the following corollary.

\begin{corollary}\label{cor:dimSubexpGrowth}
Let $G=(V,w,m,d_0)$ be a salami and $x \in V$. Then, 
\[
\dim \left\{ u \in \R^V: \Delta u = 0, \quad  \max_{B_r(x)} |u| = e^{o(r)}, \quad r \to \infty \right\} = 2.
\]
\end{corollary}

\subsection{Cheng-Yau gradient estimate}

We assume the graph $G=(V,w,m)$ has bounded geometry, i.e., there exist $0<c<C$ such that $c<m(x)<C$ and $c<w_{xy}<C$ and $\Deg(x)<C$ for all $x\sim y$. 

Since the salami is quasi-isometric to $\R,$
the volume doubling property and the Poincar\'e inequality hold as shown in the following theorem.

\begin{theorem}
Let $G=(V,w,m,d_0)$ be a salami with bounded geometry.  Then there exists $C_1$ such that \begin{equation}\label{eq:VolDoubling}|B_{2R}(x)|\leq C_1|B_{R}(x)|, \quad \forall x\in V, R>0.\end{equation}
    Moreover, there exists a constant $C_2$ such that  for any $x\in V, R>0,$ and any function $f:B_{2R}(x)\to \R,$
  \begin{equation}\label{eq:poincare}\sum_{y\in B_R(x)}|f(y)-f_R|^2\leq C_2R^2\sum_{w,z\in B_{2R}(x)}|f(w)-f(z)|^2,\end{equation} where $f_R=\frac{1}{|B_R(x)|}\sum_{B_R(x)}f.$
\end{theorem}
\begin{remark} In fact, $C^{-1}R\leq |B_R(x)|\leq CR$ for any $x\in V, R\geq 1.$
\end{remark}
\begin{proof}
Inequality \eqref{eq:VolDoubling} follows from Theorem~\ref{thm:quasiIsom} and bounded geometry.
Inequality \eqref{eq:poincare} follows from \cite[Proposition~3.33]{barlow2017random} and Theorem~\ref{thm:quasiIsom}. This finishes the proof.
\end{proof}

Indeed, volume doubling and Poincar\'e inequality are also satisfied when assuming a strong version of non-negative Bakry-\'Emery curvature  \cite{horn2014volume}, and volume doubling also holds true when assuming the usual Bakry-\'Emery curvature curvature assumption \cite{munch2019li}.

It is well-known that the combination of the volume doubling property and the Poincar\'e inequality is equivalent to the parabolic Harnack inequality, see e.g. \cite{delmotte1999parabolic}.
 For the elliptic version, we have the following Harnack inequality as a consequence of \cite[Theorem~1.7]{delmotte1999parabolic} in combination with \eqref{eq:VolDoubling} and \eqref{eq:poincare}.
\begin{theorem}\label{thm:Harnack}
Let $G=(V,w,m,d_0)$ be a salami with bounded geometry. Then there exists $C$ such that for any positive harmonic function $f$ on $B_{2R}(x),$ $R\geq 1,$ we have $$\sup_{B_R(x)}f\leq C\inf_{B_R(x)}f.$$
\end{theorem}

By combining the maximum principle for the gradient (Corollary~\ref{cor:maxPrinciple}) with the elliptic Harnack inequality above and Caccioppoli inequality we now prove the following Cheng-Yau type gradient estimate for salamis.

\begin{theorem}\label{thm:ChengYau}
 Let $G=(V,w,m,d_0)$ be a salami with bounded geometry. Then there exists a constant $C$ such that for any positive harmonic function $u$ on $B_R(x_0),$ 
$$|\nabla u|(x_0)\leq \frac{C}{R} u(x_0).$$
\end{theorem}
\begin{proof} Let $f \in \mathcal H(P)$ for a connected salami partition $P.$ W.l.o.g., assume that $f(x_0)=0.$ For any $r\geq 1,$ set
$$C_r:=f^{-1}((-r,r)),\quad A_r:=f^{-1}((-r-1,-r]\cup [r,r+1)).$$ By Lemma~\ref{lem:levelsetConnected}, 
\[|A_r| \leq C, \quad \forall r\geq 1.\] By the quasi-isometry property, there exists a constant $c$ such that
$$C_r\subset B_{cr}(x_0),\quad \forall r\geq 1.$$%where $C=\frac 1 \eps \sum_{\substack{f(y)<0\\f(x)\geq 0}}  w(x,y) (f(x)-f(y)).$
For the subset $C_r,$ the boundary of $C_r$ is contained in $A_r.$ 

W.l.o.g., let $u$ be a positive harmonic function on $B_{4cR}(x_0)$ for $R\in \N.$ By the maximum principle for $|\nabla u| $ from Corollary~\ref{cor:maxPrinciple},
$$|\nabla u|^2(x_0)\leq \max_{A_r}|\nabla u|^2,\quad \forall r\geq 1.$$ Hence,
$$|\nabla u|^2(x_0)\leq \sum_{A_r}|\nabla u|^2.$$ 
We remark that in the summations, we always omit edge weight and vertex measure as we have bounded geometry. 
Summing over $r=1,2,\cdots, R,$ for $R\in \N,$ we have
$$R|\nabla u|^2(x_0)\leq \sum_{r=1}^R\sum_{A_r}|\nabla u|^2\leq \sum_{C_R}|\nabla u|^2\leq \sum_{B_{cR}(x_0)}|\nabla u|^2.$$ By the Caccioppoli inequality (see \cite[Theorem~3.1]{hua2014q}),
$$\sum_{B_{cR}(x_0)}|\nabla u|^2\leq \frac{C}{R^2}\sum_{B_{2cR}(x_0)}u^2.$$ By the Harnack inequality (Theorem~\ref{thm:Harnack}) on $B_{4cR}(x_0),$
for any $x\in B_{2cR}(x_0),$ $$u(x)\leq C u(x_0).$$ Hence,
\begin{eqnarray*} R|\nabla u|^2(x_0)&\leq&\frac{C}{R^2}\sum_{B_{2cR}(x_0)}u^2\leq\frac{CR}{R^2}u^2(x_0).\\
\end{eqnarray*} This proves the result.
\end{proof}

\section{Examples} \label{sec:Examples}

In this section, we give examples justifying our assumptions of infinite measure of the ends, and of lower bounded edge weights in some applications. 
Moreover, we give examples of salamis with unbounded edge weights and show why the minimal cut approach fails to prove the main theorem.
Finally, we show that every birth death chain with bounded edge weight can be equipped with a path distance so that the curvature is non-negative.

\subsection{Why infinite measure?}

We now give an example showing that requiring infinite measure on two ends is necessary to conclude that there are at most two 
ends.

\begin{example}\label{ex:threeEnds}
The following graph has non-negative Ricci curvature, three ends and infinite measure.
The graph $G=(V,w,m,d_0)$ is constructed by gluing together three infinite birth death chains $G_i=(V_i=\N_0,w_i,m_i)$ by identifying their roots $0 \in \N_0$.
The chain $G_1$ is the uniform birth death chain given by $w_1(n,n+1) = 1$, and $m_1(n)=1$ for all $n \in \N_0$.
The other two chains $G_2$ and $G_3$ are set to be equal and given by
$w_2(n,n+1) = \frac 1 2 \cdot 3^{-n}$ and $m_2(n) = 3^{-n}$ for $n \in \N_0$. If $|n-m|\neq 1$, then $w_i(n,m) := 0$ for $i=1,2,3$.
Clearly, $G$ has three ends and infinite measure. 
\end{example}
We now show that $G$ has curvature $0$ everywhere.
\begin{lemma}
The graph $G$ from Example~\ref{ex:threeEnds} has Ollivier curvature $0$ on all edges.
\end{lemma}
\begin{proof}
We recall $q(x,y) = \frac{w(x,y)}{m(x)}$ for $x,y \in V$.
As $G$ is a tree, Lemma~\ref{lem:CurvatureTrees} shows the curvature $\kappa(x,y)$ for $x \sim y$ is given by
\[
\kappa(x,y) = 2(q(x,y) + q(y,x)) - \Deg(x)-\Deg(y).
\]

We notice $q(n,n+1)=q(n,n-1)=1$ on $G_1$, and $q(n,n+1)=\frac 1 2$ and $q(n,n-1)=\frac 3 2$ on $G_2$ and $G_3$
Hence on $G$, we have $\Deg \equiv 2 = q(x,y)+q(y,x)$ for all $x \sim y$. Therefore, $\kappa(x,y)=0$ for all $x \sim y$.
This finishes the proof.
\end{proof}

\subsection{Why edge weight bounded from below?}

We now give an example showing that the uniform size bound of the level sets of functions in $\mathcal H_0$ from Lemma~\ref{lem:levelsetSize} is wrong if we omit the lower bound on the edge weight.
The example is a salami on two opposite quadrants of $\Z^2$, see Figure~\ref{fig1}. We first calculate the curvature on a weighted lattice. For convenience, we use the notation of Gaussian integers.

\begin{lemma}\label{lem:curvatureLattice}
Let $V \subseteq \Z^2\subset \C$ and $G=(V,w,m,d_0)$ be a graph such that $d_0(x,y) = \|x-y\|_1$ for all $x,y \in V$. 
Let $x \in V$ and assume $y = x+1 \in V$.
Then,
\begin{align*}
\kappa(x,y) = & \quad q(x,y) + q(y,x) - q(x,x-1)-q(y,y+1)  \\& - |q(x,x+i)-q(y,y+i)| - |q(x,x-i)-q(y,y-i)|.
\end{align*}
\end{lemma}

\begin{proof}
We first prove $``\geq"$.
Let $f \in Lip(1)$ with $f(y)-f(x)=1$. W.l.o.g., $f(y)=1$ and $f(x)=0$.
Then,
\begin{align*}
\Delta f(x) &= q(x,y) + q(x,x-1)f(x-1) + q(x,x+i)f(x+i) + q(x,x-i)f(x-i) \\&
\geq  q(x,y) - q(x,x-1) + q(x,x+i)f(x+i) + q(x,x-i)f(x-i)
\end{align*}
and similarly,
\begin{align*}
\Delta f(y) &\leq -q(y,x) + q(y,y+1) + q(y,y+i) (f(y+i)-1) + q(y,y-i)(f(y-i)-1) \\
&\leq  -q(y,x) + q(y,y+1) + q(y,y+i) f(x+i) + q(y,y-i)f(x-i).
\end{align*}
Thus,
\begin{align*}
\Delta f(x)-\Delta f(y) \geq & \quad q(x,y) + q(y,x) - q(x,x-1)-q(y,y+1)  \\&
+f(x+i)(q(x,x+i) - q(y,y+i)) + f(x-i)(q(x,x-i)-q(y,y-i)) \\
 \geq  &
\quad q(x,y) + q(y,x) - q(x,x-1)-q(y,y+1)  \\& - |q(x,x+i)-q(y,y+i)| - |q(x,x-i)-q(y,y-i)|
\end{align*}
where the last inequality follows as $f(x\pm i) \in [-1,1]$.
This proves $``\geq "$. As all inequalities are sharp with an appropriate choice of $f$, we also get $``\leq"$.
\end{proof}

We now provide a salami with no uniform bound on the level set size of functions in $\mathcal H_0$.
The example is constructed by gluing two identical parts together at their base vertex. Each part is a kind of twisted cartesian product of two birth death chains.

\begin{example}\label{ex:edgeweightnotlowerbounded}
Let $X=(\N_0,w_X,m_X,d_0)$ be the birth death chain on $\N_0$ given by $w_X(n,n+1)=2^n,$  and $m_X(n)=2^{n-1}\vee 1$ for all $n\in \N_0.$ Hence $q_X(n,n-1)=1$ for $n \geq 1$ and $q_X(n,n+1) = 1 + 1_{n\geq 1}$ for $n\geq 0$.
By Lemma~\ref{lem:CurvatureTrees}, we have $\kappa_X(n,n+1)=0$ for all $n\geq 0$.

Let  $Y=(\N_0,w_Y,m_Y,d_0)$ be the birth death chain on $\N_0$ given by $w_Y(n,n+1)=2^{-n} = m_Y(n)$  for all $n\in \N_0.$ Hence $q_Y(n,n-1)=2$ for $n \geq 1$ and $q_Y(n,n+1) = 1$ for $n\geq 0$.
By Lemma~\ref{lem:CurvatureTrees}, we have $\kappa_Y(n,n+1)=2\cdot1_{n=0}$ for all $n\geq 0$.

Let $G_0:= (V_0=\N_0^2,w_0,m_0) := X \times Y$, i.e., $$w_0((x_1,y_1),(x_2,y_2)) = w_X(x_1,x_2)1_{y_1=y_2}m_Y(y_1) + w_Y(y_1,y_2)1_{x_1=x_2}m_X(x_1)$$ and $m_0(x_1,y_1) = m_X(x_1)m_Y(y_1)$ for all $x_1,y_1,x_2,y_2 \in \N_0$.
Hence, 
\[
q_0((x_1,y_1),(x_2,y_2)) = q_X(x_1,x_2)1_{y_1=y_2} + q_Y(y_1,y_2)1_{x_1=x_2}.
\]
Let $\phi:\N_0^2 \to \N_0^2, (x,y) \mapsto (\min(x,y),\max(x,y))$.
Let $G=(V,w,m,d_0)$ be given by $V=\N_0^2,$ and
$m(x,y):=m_0(\phi(x,y))$, and 
\[w((x_1,y_1),(x_2,y_2)) := w_0(\phi(x_1,y_1),\phi(x_2,y_2)) \cdot 1_{(x_1-y_1)(x_2-y_2) \geq 0}.\]
We notice that $(x_1,y_1) \sim (x_2,y_2)$ in $G$ if and only if $(x_1,y_1) \sim (x_2,y_2)$ in $G_0$.
We now claim that $\kappa \geq 0$ on $G$ with the combinatorial distance $d=d_0$.
We notice that $G$ is invariant under interchanging the role of $x$ and $y$.
Let $p_1=(x_1,y_1) \sim p_2 = (x_2,y_2)$. We aim to show $\kappa(p_1,p_2) \geq 0$. W.l.o.g., $y_1 \geq x_1$ and $y_2 \geq x_2$, and $y_2 \geq y_1$ and $x_2 \geq x_1$.
We first assume $(x_1,y_1) \neq 0$.

Case 1: $x_1=x_2=x$ and $y_2= y_1+1$.
If $x_1 < y_1$, then $G$ is isomorphic to $G_0$ on $B_1(p_1) \cup B_1(p_2)$ giving $\kappa(p_1,p_2) \geq 0$.
If $x= x_1 = y_1$, then let 
\begin{align*}
p_0&:=(x,x - 1), \\
p_3&:=(x,x + 2), \\
r_i &:=(x + 1, y_i), \\
s_i&:= (x-1,y_i),\quad i=1,2.
\end{align*}
We observe
\begin{align*}
q(p_1,p_0) &=q_X(x,x-1)= 1,   \\
q(p_1,p_2) &= q_Y(x,x+1)=1,   \\
q(p_1,r_1) &= q_Y(x,x+1)=1,   \\
q(p_1,s_1) &= q_X(x,x-1) = 1,   \\
q(p_2,p_1) &= q_Y(x+1,x) =2,    \\
q(p_2,p_3) &= q_Y(x+1,x+2)= 1,   \\
q(p_2,r_2) &=  q_X(x,x+1) = 2,   \\
q(p_2,s_2) &= q_X(x,x-1)  = 1.
\end{align*}
We give a transport plan $\rho:B_1(p_1) \times B_1(p_2) \to [0,\infty)$ by 
\[
\rho(r_1,r_2) = \rho(s_1,s_2) = \rho(p_0,p_1) = \rho(p_2,p_3) = \rho(p_1,p_1) = \rho(p_1,r_2) = 1
\]
and $\rho(p,q)=0$ otherwise.
One can easily check that $\sum_{v,w} (1-d(v,w))\rho(v,w) = 0$
which implies $\kappa(p_1,p_2) \geq 0$ by \cite[Proposition~2.4]{munch2017ollivier}. Alternatively, one can also show $\kappa(p_1,p_2)= 0$ via Lemma~\ref{lem:curvatureLattice}.

Case 2: $x_2=x_1+1$ and $y_2= y_1 = y$.
Similar to the first case, we can assume $y=x_2$.
This however means considering the edge $(y-1,y) \sim (y,y)$ for which the curvature calculation is similar to the first case.

We finally consider $\kappa((0,0),(0,1))$, and the curvature calculation is similar to above with the following exceptions. First, $q_X(x,x+1)$ is now one and not two. Hence $(p_1,r_2)$ can be omitted from the transport plan, increasing its cost by one. Second, there is no $p_0$ anymore meaning, we can now take $\rho(p_1,p_1)=2$, also increasing the transportation cost by one. Third, there is no $s_i$ anymore, and omitting does not impact the transport cost.
Hence $\kappa((0,0),(0,1))=2$.

We now construct $\widetilde G$ by gluing together
two copies of $G$ by identifying their vertex $(0,0)$.
The curvature stays same except at the edges containing $(0,0)$. 
We notice that the curvature can decrease at most by $\Deg(0,0)$ w.r.t. $G$ when gluing.
As the curvature at the corner equals the vertex degree, we see that $\kappa \geq 0$ also after gluing.
Moreover, $m(n,n) = \frac 1 2$ for all $n \in \Z \setminus \{0\}$ showing that both ends have infinite measure. Hence, $\widetilde G$ is a salami.

We construct $f \in \mathcal H_0$ on one part of $\widetilde G$ via $f(x,y)=x+y$, and on the other part as $f(x,y)=-x-y$. It is easy to check that $f \in \mathcal H_0$ and that $|f^{-1}([n,n+1))|  \to \infty$ as $n \to \infty$. This gives the counterexample to Lemma~\ref{lem:levelsetSize}
 when omitting the lower bound on the the edge weight.
\end{example}

\begin{figure}[htbp]
 \begin{center}
   \begin{tikzpicture}
    \node at (0,0){\includegraphics[width=0.4\linewidth]{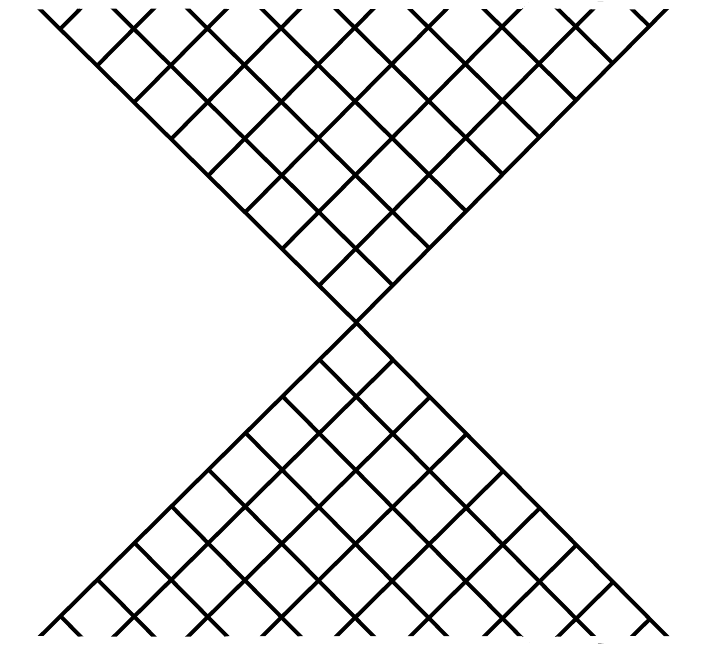}};
   % \node at (0, -.8){\Large $e$};
    \node at (2.1,   -1.4){\Large $G$};
    \node at (2.1,   1.4){\Large $G$};
  %  \node at (0,  1.3){\Large $\sigma_1$};
  %  \node at (0,  -1.7){\Large $\sigma_2$};
   % \node at (-3.4,  1.5){\Large $z$};
   \end{tikzpicture}
  \caption{The figure for Example~\ref{ex:edgeweightnotlowerbounded} shows a salami without lower bound on the edge weight.}\label{fig1}
 \end{center}
\end{figure}

\subsection{Why no edge weight upper bound on all edges?}
In Lemma~\ref{lem:wUpperBound} it is shown that for salamis, $w(x,y)$ is globally upper bounded as soon as $f(y) \neq f(x)$ for some $f \in \mathcal H_0$. We now give an example showing that we do not get an upper bound on $w(x,y)$ for all $x,y$.

\begin{example}\label{ex:noGlobalEdgeWeightBound}
Let $w : \Z \to [0,\infty)$.
Let $G=(V,w,m,d_0)$ be given by
\[
V := \{(x,y) \in \Z^2: |x-y| \leq 1\}
\]
and $m(v)=1$ for all $v \in V$ and 
\[
w((x_1,y_1),(x_2,y_2)) := \begin{cases}
1&: |x_1-x_2| + |y_1-y_2| = 1,\\
w(n)&: (x_1,y_1, x_2,y_2)  \in \{ (n,n+1,n+1,n),(n+1,n,n,n+1)  \},\\
0&: \mbox{ otherwise.}
\end{cases}
\]
We now show that $G=(V,w,m,d_0)$ is a salami. It is clear that $G$ has two ends with infinite measure. We now show non-negative curvature.

Let $(x_1,y_1)\sim (x_2,y_2)$. We first assume $|x_1-x_2| + |y_1-y_2| = 1$.
W.l.o.g., $(x_1,y_1,x_2,y_2) = (0,0,0,1)$.
Let $f \in Lip(1)$ with $f(0,0)=0$ and $f(0,1)=1$.
Then, $f(0,-1) \geq -1$ and $f(-1,0) \geq -1$. Hence, $\Delta f(0,0) \geq f(1,0)-1$.
On the other hand, $f(1,1) \leq f(1,0) + 1$ and $f(1,0)\leq 1$ and thus, $\Delta f(0,1) = f(1,1)-2 + w_0 (f(1,0)-1) \leq f(1,0) - 1 =  \Delta f(0,0)$ showing $\Delta f(0,0)- \Delta f(0,1) \geq 0$ proving $\kappa((0,0),(0,1)) \geq 0$.

For the second case, we assume $(x_1,y_1,x_2,y_2) = (n,n+1,n+1,n)$ and $w_n>0$. W.l.o.g., $n=0$, and hence $(x_1,y_1) = (0,1)$ and $(x_2,y_2)=(1,0)$. Let $f \in Lip(1)$ with $f(0,1)=0$ and $f(1,0)=0$. Then, $\Delta f(1,0) \geq f(0,0)+f(1,1) \geq \Delta f(0,1)$ showing $\kappa((0,1),(1,0)) \geq 0$.

Therefore, $G$ has non-negative curvature which shows that $G$ is a salami. The function $f:V\to \R$, $(x,y) \mapsto x+y$ belongs to $\mathcal H_0$. So we see that $w(p,q)$ is potentially unbounded exactly for the edges $(p,q)$ with $f(p)=f(q)$.
\end{example}

\begin{figure}[htbp]
 \begin{center}
   \begin{tikzpicture}
    \node at (0,0){\includegraphics[width=0.6\linewidth]{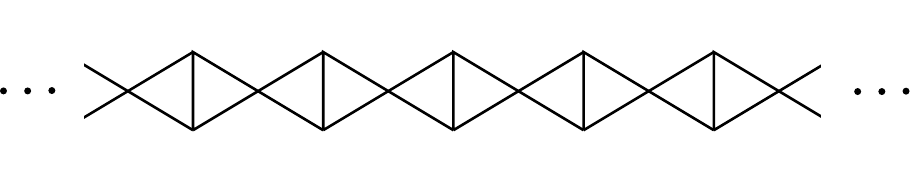}};
   % \node at (0, -.8){\Large $e$};
  %  \node at (-1.8,   -1){\Large s};
 %   \node at (2,   -1){\Large s};
  %  \node at (0,  1.3){\Large $\sigma_1$};
  %  \node at (0,  -1.7){\Large $\sigma_2$};
   % \node at (-3.4,  1.5){\Large $z$};
   \end{tikzpicture}
  \caption{The figure for Example~\ref{ex:noGlobalEdgeWeightBound} shows a salami where the weight for the vertical edges can be chosen arbitrarily.}\label{fig2}
 \end{center}
\end{figure}

\subsection{Why no minimal cut approach?}
It was mentioned in the introduction that one cannot construct functions in $\mathcal H_0$ as the distance to a minimal cut. The reason is that the distance function always takes integer values, however it can happen that there are no integer-valued functions in $\mathcal H_0$ as shown in the following example.
\begin{example} \label{ex:noIntegerH0}
Let $G=(V,w,m,d_0) $ with
$V= \Z$ and $m \equiv 1$ and
\[
w(x,y) = \begin{cases}
1&: 1 \leq |x-y| \leq 2, \\

0&: \mbox{otherwise}.
\end{cases}
\]
Then, $G$ obviously has two ends of infinite measure. Moreover, $G$ is an abelian Cayley graph with generator set $\{-2,-1,1,2\}$ and therefore has non-negative Ollivier curvature, see \cite[Theorem~3.4]{cushing2021curvatures}. Hence, $G$ is a salami.
Moreover, $f: V \to \R$, $x \mapsto \frac 1 2 x$ belongs to $\mathcal H_0$. showing that there are no integer-valued functions in $\mathcal H_0$ by Theorem~\ref{thm:H0Characterization}.
\end{example}

\begin{figure}[htbp]
 \begin{center}
   \begin{tikzpicture}
    \node at (0,0){\includegraphics[width=0.6\linewidth]{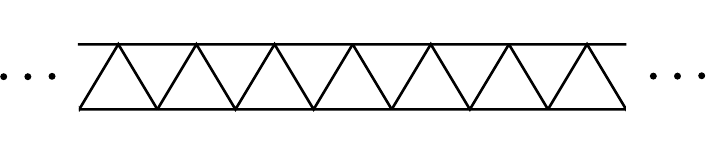}};
   % \node at (0, -.8){\Large $e$};
 %   \node at (-1.8,   -1){\Large s};
 %   \node at (2,   -1){\Large s};
  %  \node at (0,  1.3){\Large $\sigma_1$};
  %  \node at (0,  -1.7){\Large $\sigma_2$};
   % \node at (-3.4,  1.5){\Large $z$};
   \end{tikzpicture}
  \caption{The figure for Example~\ref{ex:noIntegerH0} shows a graph without integer-valued functions in $\mathcal H_0$.}\label{fig3}
 \end{center}
\end{figure}

\subsection{Why path metrics?}
We show that every non-degenerate birth death chain is a salami when equipped with the right path distance.
\begin{proposition}
Let $G=(\Z,w,m)$ be a birth death chain, i.e., $w(x,y)=0$ whenever $|x-y| \neq 1$. Assume that $w$ is bounded.
Moreover assume $m(\N) = m(-\N)=\infty$.
 Let $d : \Z^2 \to [0,\infty)$ be the path metric given by $d(n,n+1) = \frac{1}{w(n,n+1)}$. Then, $(\Z,w,m,d)$ is a salami.
\end{proposition}

\begin{proof}
As $w$ is bounded, we see that all balls are finite. By assumption, $G$ has two ends with infinite measure.
Let $n \in \Z$ and let $f\in Lip(1)$ with $f(n+1)-f(n)=d(n,n+1)$.
Then,
\[
\Delta f(n) \geq w(n,n+1)d(n,n+1) - w(n,n-1)d(n,n-1)=0
\]
and $\Delta f(n+1) \leq 0$ similarly. Hence, $\kappa(n,n+1) \geq 0$.
Thus, $(\Z,w,m,d)$ is a salami. This finishes the proof.
\end{proof}

\section*{Acknowledgments}
Both authors want to thank TSIMF in Sanya for their hospitality. 
F. M. wants to thank Mark Kempton, Gabor Lippner, Shing-Tung Yau, Slava Matveev, and J\"urgen Jost for useful discussions, and Caimeng Liu for suggesting the title of the article.
B. H. is supported by NSFC, grants no. 11831004 and no. 11926313.

\printbibliography

	Bobo Hua,\\
	School of Mathematical Sciences, LMNS, Fudan University, Shanghai 200433, China;  Shanghai Center for Mathematical Sciences, Jiangwan Campus, Fudan University, No. 2005 Songhu Road, Shanghai 200438, China. \\
	\texttt{bobohua@fudan.edu.cn}\\

	Florentin Münch, \\
	MPI MiS Leipzig, 04103 Leipzig, Germany\\
	\texttt{florentin.muench@mis.mpg.de}\\

\end{document}